\documentclass[11pt]{article}
\usepackage[normalem]{ulem}
\usepackage[margin=1in,font=scriptsize]{caption}
\usepackage{amsthm, amsmath, tikz, tkz-graph}
\usepackage{fullpage, hyperref,standalone}
\RequirePackage{marginnote,hyperref}
\newcommand{\aside}[1]{\marginnote{\scriptsize{#1}}[0cm]}
\newcommand{\aaside}[2]{\marginnote{\scriptsize{#1}}[#2]}
\newcommand\Emph[1]{\emph{#1}\aside{#1}}
\newcommand\EmphE[2]{\emph{#1}\aaside{#1}{#2}}

\newtheorem{lem}{Lemma}
\newtheorem{conj}{Conjecture}
\newtheorem*{struct-lem}{Structural Lemma}
\newtheorem{cor}[lem]{Corollary}
\newtheorem*{main}{Main Theorem}
\newcommand\Bconf{B_{conf}}

\newcommand\C{\mathcal{C}}

\renewcommand\O{\mathcal{O}}
\renewcommand\S{\mathcal{S}}
\newcommand\W{\mathcal{W}}
\newcommand\vph{\varphi}
\newcommand\card[1]{|#1|}
\newcommand\ceil[1]{\lceil#1\rceil}
\newcommand\floor[1]{\lfloor#1\rfloor}
\newcommand\irange[1]{[#1]}
\newcommand\nf{n_{5}}
\newcommand\Mns{n_{6}}

{

\DeclareRobustCommand\SMALLvert{%
\begin{tikzpicture}
\tikzstyle{tVert}=[shape = circle, minimum size = 5pt, inner sep = 0pt, outer sep = 0pt, draw]
\node[tVert]{};
\end{tikzpicture}
}

\DeclareRobustCommand\Tvert{%
\begin{tikzpicture}
\tikzstyle{SVert}=[shape = circle, minimum size = 5pt, inner sep = 0pt, outer sep = 0pt, draw, fill]
\node[SVert]{};
\end{tikzpicture}
}

\DeclareRobustCommand\BIGvert{%
\begin{tikzpicture}
\tikzstyle{BVert}=[shape = rectangle, minimum size = 5pt, inner sep = 0pt, outer
sep = 0pt, draw, fill]
\node[BVert]{};
\end{tikzpicture}
}

\begin{document}
\title{Acyclic edge-coloring of planar graphs:\\ $\Delta$ colors suffice when
$\Delta$ is large}
\author{Daniel W. Cranston\thanks{Department of Mathematics and Applied
Mathematics, Viriginia Commonwealth University, Richmond, VA;
\texttt{dcranston@vcu.edu}; 
This research is partially supported by NSA Grant H98230-15-1-0013.}
}
\maketitle
\begin{abstract}
An \emph{acyclic edge-coloring} of a graph $G$ is a proper edge-coloring of $G$
such that the subgraph induced by any two color classes is acyclic.
The \emph{acyclic chromatic index}, $\chi'_a(G)$, is the smallest number of
colors allowing an acyclic edge-coloring of $G$.
Clearly $\chi'_a(G)\ge
\Delta(G)$ for every graph $G$.  Cohen, Havet, and M\"{u}ller conjectured that
there exists a constant $M$ such that every planar graph with $\Delta(G)\ge M$
has $\chi'_a(G)=\Delta(G)$.  We prove this conjecture.
\end{abstract}

\section{Introduction}

A \EmphE{proper edge-coloring}{-4mm} of a graph $G$ assigns colors to the edges of $G$
such that two edges receive distinct colors whenever they have an endpoint in
common.  An \EmphE{acyclic edge-coloring}{2mm} is a proper edge-coloring such that the
subgraph induced by any two color classes is acyclic (equivalently, the edges of
each cycle receive at least three distinct colors).  
The \EmphE{acyclic chromatic index}{3.5mm}, $\chi'_a(G)$, is the smallest number of
colors allowing an acyclic edge-coloring of $G$.
In an edge-coloring $\vph$, if a color $\alpha$ is used incident to a vertex
$v$, then $\alpha$ is \EmphE{seen by}{8mm} $v$.
For the maximum degree of $G$, we write $\Delta(G)$, and simply $\Delta$ when
the context is clear.  Note that $\chi'_a(G)\ge \Delta(G)$ for every graph $G$. 
When we write \emph{graph}, we forbid loops and multiple edges.  A \EmphE{planar
graph}{2mm} is one that can be drawn in the plane with no edges crossing.
A \EmphE{plane graph}{5mm} is a planar embedding of a planar graph.
Cohen, Havet, and M\"{u}ller~\cite{CHM, BCCHM} conjectured that there exists a
constant $M$ such that every planar graph with $\Delta(G)\ge M$ has
$\chi'_a(G)=\Delta(G)$.  We prove this conjecture.

\begin{main}
All planar graphs $G$ satisfy $\chi'_a(G)\le \max\{\Delta,4.2*10^{14}\}$.
Thus, $\chi'_a(G)=\Delta$ for all planar graphs $G$ with $\Delta\ge 4.2*10^{14}$.
\end{main}

We start by reviewing the history of acyclic coloring and
acyclic edge-coloring. 
An \Emph{acyclic coloring} of a graph $G$ is a proper vertex coloring of $G$
such that the subgraph induced by any two color classes is acyclic.  The
smallest number of colors that allows an acyclic coloring of $G$ is the
\Emph{acyclic chromatic number}, $\chi_a(G)$.  This concept was introduced in
1973 by Gr\"{u}nbaum~\cite{grunbaum-acyclic}, who conjectured that
every planar graph $G$ has $\chi_a(G)\le 5$.  This is best possible, as shown
(for example) by the octahedron.  After a flurry of activity, Gr\"{u}nbaum's
conjecture was confirmed in 1979 by Borodin~\cite{borodin-acyclic}.  
This result contrasts sharply with the behavior of $\chi_a(G)$ for a general
graph $G$.  Alon, McDiarmid, and Reed~\cite{AMR} found a constant $C_1$
such that for every $\Delta$ there exists a graph $G$ with maximum degree
$\Delta$ and $\chi_a(G)\ge C_1 \Delta^{4/3}(\log \Delta)^{-1/3}$.  This
construction is nearly best possible, since they also found a constant $C_2$
such that $\chi_a(G)\le C_2\Delta^{4/3}$ for every graph $G$ with maximum degree
$\Delta$.
The best known upper bound is $\chi_a(G)\le 2.835\Delta^{4/3}+\Delta$, due to
Sereni and Volec~\cite{SV}.


Now we turn to acyclic edge-coloring.
In contrast to the results above,
there does exist a constant $C_3$ such that $\chi'_a(G)\le
C_3\Delta$ for every graph $G$ with maximum degree $\Delta$.
Using the Asymmetric Local Lemma, Alon, McDiarmid, and Reed~\cite{AMR} 
showed that we can take $C_3=64$.  This constant has been improved
repeatedly, and the 
current best bound is $2$, due to Kirousis and Livieratos~\cite{KL}.  
But this upper bound is still far from the conjectured actual value.

\begin{conj}
\label{conj1}
Every graph $G$ satisfies $\chi'_a(G)\le \Delta+2$.
\end{conj}

Conjecture~\ref{conj1} was posed by Fiam\v{c}\'{i}k~\cite{fiamcik} in 1978, and again
by Alon, Sudakov, and Zaks~\cite{ASZ} in 2001.
The value $\Delta+2$ is best possible, as shown (for example) by $K_n$ when $n$
is even. In an acyclic edge-coloring at most one color class can be a perfect
matching; otherwise, two perfect matchings will induce some cycle, by the
Pigeonhole principle.  Now the lower bound $\Delta+2$ follows from an easy counting
argument.

For planar graphs, the best upper bounds are much closer to the conjectured
value.  Cohen, Havet, and M\"{u}ller~\cite{CHM}
proved $\chi'_a(G)\le \Delta+25$ whenever $G$ is planar.  The constant 25 has
been frequently improved~\cite{BC, BCCHM, GHY, WSW, WZ}.  The current best bound is
$\chi'_a(G)\le \Delta+6$, due to Wang and Zhang~\cite{WZ}.  However, for planar
graphs with $\Delta$ sufficiently large, Conjecture~\ref{conj1} can be 
strengthened further.  This brings us to the previously mentioned conjecture of
Cohen, Havet, and M\"{u}ller~\cite{CHM}.

\begin{conj}
\label{conj2}
There exists a constant $M$ such that if $G$ is planar and $\Delta\ge M$, then
$\chi'_a(G)=\Delta$.
\end{conj}

Our Main Theorem confirms Conjecture~\ref{conj2}.
%
For the proof we consider a hypothetical counterexample. Among all
counterexamples we choose one with the fewest vertices, a \emph{minimal
counterexample}\aside{minimal counter- example}.  In Section~\ref{structural-sec} we prove our Structural Lemma,
which says that every 2-connected plane graph contains one of four
configurations.  In Section~\ref{reducibility-sec} we show that every minimal
counterexample $G$ must be 2-connected, and that $G$ cannot contain any of
these four configurations.  This shows that no minimal counterexample exists,
which finishes the proof of the Main Theorem.

\section{The Structural Lemma}
\label{structural-sec}

A vertex $v$ is \EmphE{big}{-3mm} if $d(v)\ge 8680$.
\aside{$k/k^+/k^-$-vertex} 
For a graph $G$, a vertex $v$ is \Emph{very big} if $d(v)\ge \Delta-4(8680)$.
A \emph{$k$-vertex} 
(resp.~\emph{$k^+$-vertex} and
\emph{$k^-$-vertex}) is a vertex of degree $k$ (resp.~at least $k$ and at most
$k$).  For a vertex $v$, a \emph{$k$-neighbor} is an adjacent $k$-vertex;
\emph{$k^+$-neighbors} and \emph{$k^-$-neighbors} are defined analogously. 
Similarly, we define \emph{$k$-faces}, \emph{$k^+$-faces}, and \emph{$k^-$-faces}.  
\aaside{$k/k^+/k^-$-face}{-7mm}
For the length of a face $f$, we write \EmphE{$\ell(f)$}{-.25mm}.

A key structure in our proof, called a \Emph{bunch}, consists of two big
vertices with many common $4^-$-neighbors that are embedded as successive
neighbors (for both big vertices); see Figure~\ref{fig:bunch} for an example. 
Let $x_0,\ldots,x_{t+1}$ denote successive neighbors of a big vertex
$v$, that are also successive for a big vertex $w$.  We require that $d(x_i)\le
4$ for all $i\in\irange{t}$, where $\irange{t}$ denotes $\{1,\ldots,t\}$.
Further, for each $i\in\irange{t+1}$, we require that the 4-cycle
$vx_iwx_{i-1}$ is not separating; so, either the cycle bounds a 4-face, or it
bounds the two 3-faces $vx_ix_{i-1}$ and $wx_ix_{i-1}$.  
For such a bunch, we call $x_1,\ldots, x_t$ its \EmphE{bunch
vertices}{-7mm}, and we call $v$ and $w$ the \Emph{parents} of the bunch.  (When we
refer to a bunch, we typically mean a maximal bunch.)
For technical reasons, we exclude $x_0$ and $x_{t+1}$ from the bunch.  
Thus, each 4-vertex in a bunch is incident to four 3-faces, each 3-vertex in a bunch
is incident to a 4-face and two 3-faces, and each 2-vertex in a bunch is
incident to two 4-faces. 
The
\Emph{length} of the bunch is $t$.  
A \emph{horizontal edge}\aaside{horizon- tal edge}{4.3mm} is any
edge $x_ix_{i+1}$, with $1\le i\le t-1$. 
Each path $vx_iw$ is a \EmphE{thread}{-.5mm}.

\begin{figure}
\centering
\begin{tikzpicture}[scale = 8]
\tikzstyle{VertexStyle} = []
\tikzstyle{EdgeStyle} = []
\tikzstyle{labeledStyle}=[shape = circle, minimum size = 15pt, inner sep = 1.2pt]
\tikzstyle{unlabeledStyle}=[shape = circle, minimum size = 6pt, inner sep = 1.2pt, draw, fill]
\Vertex[style = unlabeledStyle, x = 0.250, y = 0.650, L = \small {}]{v0}
\Vertex[style = unlabeledStyle, x = 0.350, y = 0.650, L = \tiny {}]{v1}
\Vertex[style = unlabeledStyle, x = 0.450, y = 0.650, L = \tiny {}]{v2}
\Vertex[style = unlabeledStyle, x = 0.550, y = 0.650, L = \tiny {}]{v3}
\Vertex[style = unlabeledStyle, x = 0.650, y = 0.650, L = \tiny {}]{v4}
\Vertex[style = unlabeledStyle, x = 0.150, y = 0.650, L = \small {}]{v5}
\Vertex[style = labeledStyle, x = 0.103, y = 0.650, L = \small {$x_0$}]{v55}
\Vertex[style = unlabeledStyle, x = 0.750, y = 0.650, L = \tiny {}]{v6}
\Vertex[style = unlabeledStyle, x = 0.850, y = 0.650, L = \tiny {}]{v7}
\Vertex[style = unlabeledStyle, x = 0.950, y = 0.650, L = \small {}]{v8}
\Vertex[style = unlabeledStyle, x = 1.050, y = 0.650, L = \small {}]{v9}
\Vertex[style = labeledStyle, x = 1.12, y = 0.650, L = \small {$x_{t+1}$}]{v95}
\Vertex[style = unlabeledStyle, x = 0.600, y = 0.950, L = \small {}]{v10}
\Vertex[style = labeledStyle, x = 0.600, y = 0.989, L = \small {$v$}]{v105}
\Vertex[style = unlabeledStyle, x = 0.600, y = 0.350, L = \small {}]{v11}
\Vertex[style = labeledStyle, x = 0.600, y = 0.311, L = \small {$w$}]{v115}
\Edge[label = \tiny {}, labelstyle={auto=right, fill=none}](v0)(v10)
\Edge[label = \tiny {}, labelstyle={auto=right, fill=none}](v1)(v10)
\Edge[label = \tiny {}, labelstyle={auto=right, fill=none}](v2)(v10)
\Edge[label = \tiny {}, labelstyle={auto=right, fill=none}](v3)(v10)
\Edge[label = \tiny {}, labelstyle={auto=right, fill=none}](v4)(v10)
\Edge[label = \tiny {}, labelstyle={auto=right, fill=none}](v5)(v10)
\Edge[label = \tiny {}, labelstyle={auto=right, fill=none}](v6)(v10)
\Edge[label = \tiny {}, labelstyle={auto=right, fill=none}](v7)(v10)
\Edge[label = \tiny {}, labelstyle={auto=right, fill=none}](v8)(v10)
\Edge[label = \tiny {}, labelstyle={auto=right, fill=none}](v9)(v10)
\Edge[label = \tiny {}, labelstyle={auto=right, fill=none}](v0)(v11)
\Edge[label = \tiny {}, labelstyle={auto=right, fill=none}](v1)(v11)
\Edge[label = \tiny {}, labelstyle={auto=right, fill=none}](v2)(v11)
\Edge[label = \tiny {}, labelstyle={auto=right, fill=none}](v3)(v11)
\Edge[label = \tiny {}, labelstyle={auto=right, fill=none}](v4)(v11)
\Edge[label = \tiny {}, labelstyle={auto=right, fill=none}](v5)(v11)
\Edge[label = \tiny {}, labelstyle={auto=right, fill=none}](v6)(v11)
\Edge[label = \tiny {}, labelstyle={auto=right, fill=none}](v7)(v11)
\Edge[label = \tiny {}, labelstyle={auto=right, fill=none}](v8)(v11)
\Edge[label = \tiny {}, labelstyle={auto=right, fill=none}](v9)(v11)
\Edge[label = \tiny {}, labelstyle={auto=right, fill=none}](v0)(v5)
\Edge[label = \tiny {}, labelstyle={auto=right, fill=none}](v0)(v1)
\Edge[label = \tiny {}, labelstyle={auto=right, fill=none}](v4)(v3)
\Edge[label = \tiny {}, labelstyle={auto=right, fill=none}](v4)(v6)
\Edge[label = \tiny {}, labelstyle={auto=right, fill=none}](v7)(v8)
\Edge[label = \tiny {}, labelstyle={auto=right, fill=none}](v9)(v8)
\end{tikzpicture}
\caption{A bunch, with $v$ and $w$ as its parents.\label{fig:bunch}}
\end{figure}
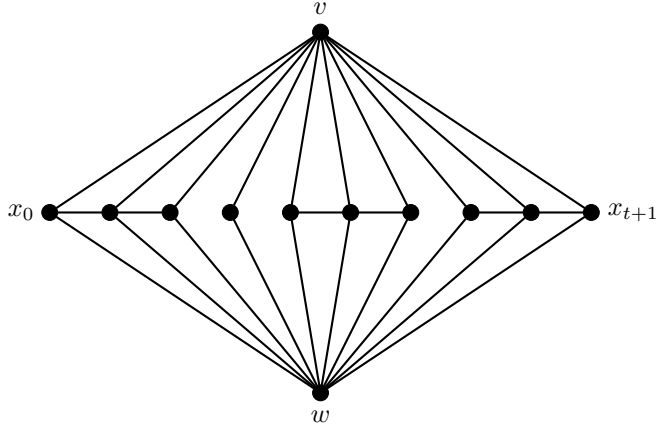

Borodin et al.~\cite{BBGvdH2} constructed graphs in which every $5^-$-vertex has at
least two big neighbors.  Begin with a truncated dodecahedron, and subdivide
$t$ times each edge that lies on two 10-faces.  Now add a new vertex into every
$4^+$-face, making it adjacent to every vertex on the face boundary.  The
resulting plane triangulation has $\Delta=5t+10$, minimum degree 4, and every
$5^-$-vertex has two $\Delta$-neighbors.  This final fact motivates our
Structural Lemma, by showing that if we omit from it (C3) and (C4), then the
resulting statement is false.  (For illustrating that we cannot omit both (C3)
and (C4), the above construction can be
generalized.  Rather than truncating a dodecahedron, we can start by truncating
any 3-connected plane graph with all faces of length 5 or 6; the rest of the
construction is the same.)  Now we state and prove our Structural Lemma.

\begin{struct-lem}
Let $G$ be a 2-connected plane graph.  Let $k=\max\{\Delta,5(8680)\}$\aside{$k$}.
Now $G$ contains one of the following four configurations:
\begin{enumerate}
\item[(C1)] a vertex $v$ such that $\sum_{w\in N(v)}d(w)\le k$; or

\item[(C2)] a big vertex $v$ such that among those
$5^-$-vertices which have $v$ as their unique big neighbor 
the number of 
(i) 2-vertices is at least 8889 or 
(ii) $3^-$-vertices is at least 17655 or
(iii) $4^-$-vertices is at least 26401 or 
(iv) $5^-$-vertices is at least 35137; or

\item[(C3)] a big vertex $v$ such that $\nf+2\Mns\le 35$, where $\nf$ and
$\Mns$\aside{$\nf, \Mns$}
denote the number of $5^-$-neighbors and $6^+$-neighbors of $v$ that are in no
bunch with $v$ as a parent; or

\item[(C4)] a very big vertex $v$ such that $\nf+2\Mns\le 141415$, where $\nf$
and $\Mns$ denote the number of $5^-$-neighbors and $6^+$-neighbors of $v$ that
are in no bunch with $v$ as a parent. 

\end{enumerate}
\end{struct-lem}

\begin{proof}
We use discharging, assigning charge $d(v)-6$ to each vertex $v$ and charge
$2\ell(f)-6$ to each face $f$.  By Euler's formula, the sum of these charges is
$-12$.  We assume that $G$ contains none of the four configurations and
redistribute charge so that each vertex and face ends with nonnegative charge,
a contradiction.  We use the following three discharging rules.
\begin{enumerate}
\item[(R1)] 
Let $v$ be a $5^-$-vertex.
If $v$ has a single big neighbor $w$, then $v$ takes $6-d(v)$ from $w$.
If $v$ is in a bunch, then $v$ takes 1 from each parent of the bunch.
If $v$ has exactly two big neighbors, and they are not its parents in a bunch,
then $v$ takes $\frac12$ from each of these big neighbors.

\item[(R2)] Let $v$ be a $5^-$-vertex with a big neighbor $w$, and let $vw$ lie
on a face $f$.  If $\ell(f)=4$, then $v$ takes 1 from $f$.  If $\ell(f)\ge 5$
and $v$ has a second big neighbor along $f$, then $v$ takes 2 from $f$.
Otherwise, if $\ell(f)\ge 5$, then $v$ takes 1 from $f$.

\item[(R3)] Every $5^-$-vertex on a 3-face with two big neighbors takes 2 from a
central ``bank''\aside{bank}; each big vertex gives 12 to the bank.
\end{enumerate}

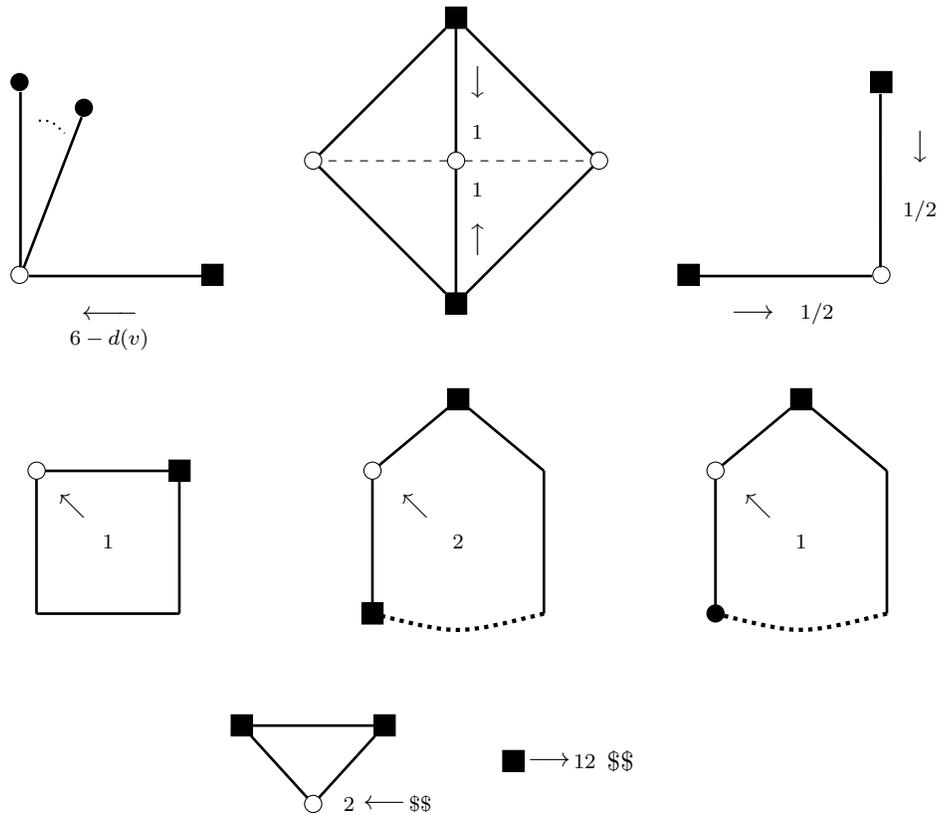
\begin{figure}

\centering
\tikzstyle{BVert}=[shape = rectangle, minimum size = 8pt, inner sep = 0pt, outer
sep = 0pt, draw, fill]
\tikzstyle{tVert}=[shape = circle, minimum size = 6.5pt, inner sep = 0pt, outer sep = 0pt, draw]
\tikzstyle{SVert}=[shape = circle, minimum size = 6.5pt, inner sep = 0pt, outer sep = 0pt, draw, fill]
\tikzstyle{uVert}=[shape = circle, minimum size = 0pt, inner sep = 0pt, outer sep = 0pt]

%

\begin{tikzpicture}[scale = 9.5]
\clip (0.05,1) rectangle (1.475,-.25);
\tikzstyle{VertexStyle} = []
\tikzstyle{EdgeStyle} = [line width=1pt]

\begin{scope}[xshift=.25cm] 
\Vertex[style = BVert, x = 0.500, y = 0.900, L = \small {}]{v0}
\Vertex[style = tVert, x = 0.500, y = 0.700, L = \small {}]{v1}
\Vertex[style = tVert, x = 0.300, y = 0.700, L = \small {}]{v2}
\Vertex[style = tVert, x = 0.700, y = 0.700, L = \small {}]{v3}
\Vertex[style = BVert, x = 0.500, y = 0.500, L = \small {}]{v4}
\Vertex[style = uVert, x = 0.530, y = 0.810, L = \small {$\big\downarrow$}]{v5}
\Vertex[style = uVert, x = 0.530, y = 0.740, L = \scriptsize {$1$}]{v55}
\Vertex[style = uVert, x = 0.530, y = 0.590, L = \small {$\big\uparrow$}]{v6}
\Vertex[style = uVert, x = 0.530, y = 0.660, L = \scriptsize {$1$}]{v55}
\Edge[label = \small {}, labelstyle={auto=right, fill=none}](v1)(v0)
\Edge[label = \small {}, labelstyle={auto=right, fill=none}](v2)(v0)
\Edge[label = \small {}, labelstyle={auto=right, fill=none}](v3)(v0)
\Edge[label = \small {}, labelstyle={auto=right, fill=none}](v1)(v4)
\Edge[label = \small {}, labelstyle={auto=right, fill=none}](v2)(v4)
\Edge[label = \small {}, labelstyle={auto=right, fill=none}](v3)(v4)
\draw[dashed](v1)--(v2);
\draw[dashed](v1)--(v3);

\begin{scope}[xshift=.43,scale=1.8,yshift=-.45cm]
\Vertex[style = BVert, x = 0.600, y = 0.900, L = \small {}]{v0}
\Vertex[style = tVert, x = 0.600, y = 0.750, L = \small {}]{v1}
\Vertex[style = BVert, x = 0.450, y = 0.750, L = \small {}]{v2}
\Vertex[style = uVert, x = 0.630, y = 0.850, L = \small {$\big\downarrow$}]{v3}
\Vertex[style = uVert, x = 0.630, y = 0.800, L = \scriptsize {$1/2$}]{v35}
\Vertex[style = uVert, x = 0.500, y = 0.720, L = \small {$\longrightarrow$}]{v4}
\Vertex[style = uVert, x = 0.550, y = 0.720, L = \scriptsize {$1/2$}]{v45}
\Edge[label = \small {}, labelstyle={auto=right, fill=none}](v0)(v1)
\Edge[label = \small {}, labelstyle={auto=right, fill=none}](v2)(v1)
\begin{scope}[xshift=.53cm,xscale=-1]
\Vertex[style=uVert, x=.640, y=.600, L = {$~$}]{}
\Vertex[style = SVert, x = 0.600, y = 0.900, L = \small {}]{v0}
\Vertex[style = SVert, x = 0.550, y = 0.88, L = \small {}]{v01}
\Vertex[style = tVert, x = 0.600, y = 0.750, L = \small {}]{v1}
\Vertex[style = BVert, x = 0.450, y = 0.750, L = \small {}]{v2}
\Vertex[style = uVert, x = 0.530, y = 0.720, L = \small
{$\longleftarrow\mkern-10mu-$}]{v4}
\Vertex[style = uVert, x = 0.530, y = 0.700, L = \scriptsize {$6-d(v)$}]{v45}
\draw[dotted, thick] (.585,.87) .. controls (.575,.87) .. (.565,.86);
\Edge[label = \small {}, labelstyle={auto=right, fill=none}](v0)(v1)
\Edge[label = \small {}, labelstyle={auto=right, fill=none}](v1)(v01)
\Edge[label = \small {}, labelstyle={auto=right, fill=none}](v2)(v1)

\end{scope}
\end{scope}
\end{scope}


\begin{scope}[yshift=-.23in,xshift=-.337cm]
\Vertex[style = tVert, x = 0.500, y = 0.850, L = \small {}]{v0}
\Vertex[style = BVert, x = 0.700, y = 0.850, L = \small {}]{v1}
\Vertex[style = uVert, x = 0.700, y = 0.650, L = \small {}]{v2}
\Vertex[style = uVert, x = 0.500, y = 0.650, L = \small {}]{v3}
\Vertex[style = uVert, x = 0.550, y = 0.800, L = \small {$\nwarrow$}]{v4}
\Vertex[style = uVert, x = 0.600, y = 0.750, L = \scriptsize {$1$}]{v45}
\Edge[label = \small {}, labelstyle={auto=right, fill=none}](v1)(v0)
\Edge[label = \small {}, labelstyle={auto=right, fill=none}](v3)(v0)
\Edge[label = \small {}, labelstyle={auto=right, fill=none}](v1)(v2)
\Edge[label = \small {}, labelstyle={auto=right, fill=none}](v3)(v2)

\begin{scope}[xshift =0.37cm, xscale=1.2] 
\Vertex[style = tVert, x = 0.500, y = 0.850, L = \small {}]{v0}
\Vertex[style = uVert, x = 0.700, y = 0.850, L = \small {}]{v1}
\Vertex[style = uVert, x = 0.700, y = 0.650, L = \small {}]{v2}
\Vertex[style = BVert, x = 0.500, y = 0.650, L = \small {}]{v3}
\Vertex[style = uVert, x = 0.550, y = 0.800, L = \small {$\nwarrow$}]{v4}
\Vertex[style = uVert, x = 0.600, y = 0.750, L = \scriptsize {$2$}]{v45}
\Vertex[style = BVert, x = 0.600, y = 0.950, L = \small {}]{v5}
\Edge[label = \small {}, labelstyle={auto=right, fill=none}](v1)(v5)
\Edge[label = \small {}, labelstyle={auto=right, fill=none}](v5)(v0)
\Edge[label = \small {}, labelstyle={auto=right, fill=none}](v3)(v0)
\Edge[label = \small {}, labelstyle={auto=right, fill=none}](v1)(v2)
\draw[dotted, ultra thick](v2) .. controls (.6,.62) .. (v3);

\begin{scope}[xshift =0.40cm] 
\Vertex[style = tVert, x = 0.500, y = 0.850, L = \small {}]{v0}
\Vertex[style = uVert, x = 0.700, y = 0.850, L = \small {}]{v1}
\Vertex[style = uVert, x = 0.700, y = 0.650, L = \small {}]{v2}
\Vertex[style = SVert, x = 0.500, y = 0.650, L = \small {}]{v3}
\Vertex[style = uVert, x = 0.550, y = 0.800, L = \small {$\nwarrow$}]{v4}
\Vertex[style = uVert, x = 0.600, y = 0.750, L = \scriptsize {$1$}]{v45}
\Vertex[style = BVert, x = 0.600, y = 0.950, L = \small {}]{v5}
\Edge[label = \small {}, labelstyle={auto=right, fill=none}](v1)(v5)
\Edge[label = \small {}, labelstyle={auto=right, fill=none}](v5)(v0)
\Edge[label = \small {}, labelstyle={auto=right, fill=none}](v3)(v0)
\Edge[label = \small {}, labelstyle={auto=right, fill=none}](v1)(v2)
\draw[dotted, ultra thick](v2) .. controls (.6,.62) .. (v3);

\end{scope}
\end{scope}
\end{scope}



\begin{scope}[yshift=-.335in,xshift=.1cm]
\Vertex[style = BVert, x = 0.350, y = 0.760, L = \small {}]{v0}
\Vertex[style = BVert, x = 0.550, y = 0.760, L = \small {}]{v1}
\Vertex[style = tVert, x = 0.450, y = 0.650, L = \small {}]{v2}
\Vertex[style = uVert, x = 0.500, y = 0.650, L = \scriptsize {$2$}]{v4}
\Vertex[style = uVert, x = 0.550, y = 0.650, L = \small {$\longleftarrow$}]{v3}
\Vertex[style = uVert, x = 0.600, y = 0.650, L = \scriptsize {\$\$}]{v5}
\Edge[label = \small {}, labelstyle={auto=right, fill=none}](v0)(v1)
\Edge[label = \small {}, labelstyle={auto=right, fill=none}](v0)(v2)
\Edge[label = \small {}, labelstyle={auto=right, fill=none}](v2)(v1)

\begin{scope}[xshift=.23cm, yshift=0.06cm] 
\Vertex[style = BVert, x = 0.500, y = 0.650, L = \small {}]{v4}
\Vertex[style = uVert, x = 0.550, y = 0.650, L = \small {$\longrightarrow$}]{v3}
\Vertex[style = uVert, x = 0.600, y = 0.650, L = \scriptsize {$12$}]{v5}
\Vertex[style = uVert, x = 0.650, y = 0.650, L = \small {\$\$}]{v4}
\end{scope}
\end{scope}
\end{tikzpicture}

\caption{Examples of rules (R1), (R2), and (R3) are shown, from top to bottom
respectively.  Big vertices
are drawn as \BIGvert\!, and $5^-$-vertices are drawn as \SMALLvert\!,
and vertices that are not big (but possibly small) are drawn as \Tvert\!.
}
\end{figure}
\vfill

If a vertex or face ends with nonnegative charge, then it ends
\Emph{happy}.  We show that each vertex and face (and the bank)
ends happy.
Let $V_{big}$\aside{$V_{big}$} denote the set of big vertices.  The number of
$5^-$-vertices that take 2 from the bank is at most $2\card{E(G[V_{big}])}$. 
Since $G[V_{big}]$ is planar, $\card{E(G[V_{big}])}<3\card{V_{big}}$.  So
the bank ends happy, since it receives $12\card{V_{big}}$ and gives away less
than this.
\vfill

\newpage

Consider a face $f$.
\begin{enumerate}
\item $\ell(f)\ge 6$.
Rather than sending charge as in (R2), suppose that
$f$ sends 1 to each incident vertex, and then each big incident vertex sends 1
to its successor (in clockwise direction) around $f$.  Now each $5^-$-vertex
incident to $f$ receives at least as much as in (R2), so $f$ sends at least as
much as in (R2), and $f$ ends happy since $2\ell(f)-6-\ell(f)\ge 0$.

\item $\ell(f)=5$.
If $f$ sends charge to at most two incident vertices, then $f$ ends happy,
since $2\ell(f)-6-2(2)= 0$.  So suppose $f$ sends charge to at least three
incident vertices.  Now two of these receive only 1 from $f$.  So $f$ again
ends happy, since $2\ell(f)-6-2-2(1)=0$.

\item $\ell(f)=4$.
Because $f$ sends charge to at most two incident vertices, it ends happy,
since $2(4)-6-2(1)=0$.  

\item $\ell(f)=3$.
The face $f$ ends happy, since it starts and ends with 0.  
\end{enumerate}


Now we consider vertices.  Since $G$ is 2-connected, it has minimum degree at
least 2, and each vertex $v$ lies on $d(v)$ distinct faces.
A $5^-$-vertex $v$ with no big neighbor would satisfy (C1),  
so each $5^-$-vertex has at least one big neighbor.
Note that if $v$ has only one big neighbor, $w$, then $v$ takes $6-d(v)$ from
$w$ by (R1), so $v$ ends happy.  Thus, we assume $v$ has at least two big
neighbors.

\begin{enumerate}
\item[1.] $d(v)=2$.
Let $w_1$ and $w_2$ denote the two big neighbors of $v$. Since $G$ is
2-connected, the path $w_1vw_2$ lies on two (distinct) faces.  If one of these
is a 3-face, then $v$ takes 2 from the bank by (R3), at least 1 from its other incident
$4^+$-face by (R2), and $\frac12$ from each big neighbor by (R1); so $v$ ends happy, since
$2-6+2+1+2(\frac12)=0$.  If one incident face $f$ is a $5^+$-face, then $v$ takes 2
from $f$ by (R2), at least 1 from its other incident $4^+$-face by (R2), and
$\frac12$ from each big neighbor by (R1); so again $v$ ends happy. 
So assume that both incident faces are 4-faces.  Now $v$ is in a bunch with its
two big neighbors, so $v$ takes 1 from each by (R1).  Thus $v$ ends with $2-6+2(1)+2(1)=0$.

\item[2.] $d(v)=3$.
If $v$ has three big
neighbors, then for each incident face $f$, either $v$ takes at least 1 from $f$
by (R2) or $v$ takes 2 from the bank by (R3), and $v$ ends happy.  So assume
$v$ has exactly two big neighbors, $w_1$ and $w_2$.  If $w_1vw_2$ lies on a
3-face, then $v$ takes 2 from the bank by (R3) and $v$ takes $\frac12$ from
each $w_i$ by (R1), and $v$ ends happy, since $3-6+2+2(\frac12)=0$.  So assume
$w_1vw_2$ lies on a $4^+$-face.  If it lies on a $5^+$-face, or if $v$ lies on
two $4^+$-faces, then $v$ receives at least 2 from its incident faces by (R2)
and $\frac12$ from each $w_i$ by (R1), and again $v$ ends happy.  
So assume $v$ lies on a 4-face with $w_1$ and $w_2$ and also on two 3-faces. 
Now $v$ is in a bunch with $w_1$ and $w_2$, so $v$ takes 1 from each, by (R1). 
Thus, $v$ ends happy, since $3-6+1+2(1)=0$.

\item[3.] $d(v)=4$.
Suppose $v$ has at least three big neighbors.  So $v$ has two big neighbors
along at least two incident faces, $f_1$ and $f_2$.  If either $f_i$ is a
3-face, then $v$ takes 2 from the bank by (R3) and ends happy.  Otherwise $v$
takes at least 1 from each of $f_1$ and $f_2$ by (R2), so $v$ ends happy.  
Assume instead that $v$ has exactly two big neighbors, $w_1$ and $w_2$. 
Suppose that $vw_1$ and $vw_2$ are incident to the same face $f$.  If $f$ is a
3-face, then $v$ takes 2 from the bank by (R3) and ends happy.  If $f$ is a
$4^+$-face, then $v$ takes at least 1 from $f$ by (R2) and at least $\frac12$
from each big neighbor by (R1), and $v$ ends happy since $4-6+1+2(\frac12)=0$. 
So assume that $w_1$ and $w_2$ do not appear consecutively among the neighbors
of $v$.  If $v$ is incident to any $4^+$-face $f$, then $v$ takes at least 1
from $f$ by (R2) and $\frac12$ from each of its big neighbors by (R1), and again
$v$ ends happy.  Thus, we assume that $v$ lies on four 3-faces.  Now $v$ is in
a bunch with $w_1$ and $w_2$, so $v$ takes 1 from each by (R1), and $v$ ends happy.

\item[4.] $d(v)=5$.
If $v$ has exactly two big neighbors, $w_1$ and $w_2$, then $v$ receives
$\frac12$ from each by (R1), and $v$ ends with $5-6+2(\frac12)=0$.  So assume
that $v$ has at least three big neighbors.  By the Pigeonhole principle, $v$
lies on at least one face, $f$, with two big neighbors.  So $v$ receives at
least 1 from either $f$ or from the bank, by (R2) or (R3).  Thus, $v$ finishes
with at least $5-6+1=0$.

\item[5.]
$d(v)\ge 6$, but $v$ is not big.  
Now $v$ ends happy, since $d(v)-6\ge 0$.

\item[6.] $v$ is a big vertex but not a very big vertex.  
Suppose that $v$ has a
$5^-$-neighbor $w$ such that $v$ is the only big neighbor of $w$.  Now
$\sum_{x\in N(w)}d(x)\le 4(8680)+d(v)\le 4(8680)+(\Delta-4(8680))=\Delta\le k$.
 Thus, $w$ is an instance of (C1), a contradiction.  So $v$ has no such
$5^-$-neighbor.  As a result, $v$ sends at most 1 to each of its neighbors.  
Since $G$ has no instance of (C3), we have $\nf+2\Mns\ge 36$, where
$\nf$ and $\Mns$ are defined as in (C3).  Note that $v$ sends at most
$\frac12$ to each vertex counted by $\nf$ and sends no charge to each vertex
counted by $\Mns$.  Further, $v$ sends at most 1 to each other neighbor.  Also,
$v$ sends 12 to the bank.  So $v$ finishes with at least $d(v)-6-12 -\frac12\nf
-1(d(v)-\nf-\Mns) = -18+ \frac12(\nf+2\Mns)\ge -18 + \frac12(36)=0$.  

\item[7.]
$v$ is a very big vertex.  
Let $\W$ denote the set of $5^-$-vertices $w$ for which $v$ is the only big
neighbor of $w$.  Since $G$ has no instance of (C2), the numbers of
2-vertices, $3^-$-vertices, $4^-$-vertices, and $5^-$-vertices in $\W$ are
(respectively) at most 8888, 17654, 26400, and 35136.  So the total charge that
$v$ sends to these vertices is at most $8888+17654+26400+35136=88078$.  Since
$G$ has no copy of (C4), we have $\nf+2\Mns\ge 141416$.
We may assume that as many as possible of the vertices counted by $\nf$ are in
$\W$, since this is the situation in which $v$ sends the most charge.
If $w$ is counted by $\nf$ and is not in $\W$, then $v$ sends $w$ at most
$\frac12$.  If $w$ is counted by $\Mns$, then $v$ sends $w$ nothing.
So $v$ ends happy, since
$d(v)-6-12-88078-\frac12(\nf-35136)-1(d(v)-\nf-\Mns)=-18-88078+\frac12\nf+\frac12(35136)+\Mns=-70528+\frac12(\nf+2\Mns)\ge
-70528+\frac12(141416)\ge 0$.
%
\qedhere
\end{enumerate}
\end{proof}

\section{Reducibility}
\label{reducibility-sec}
In this section we use the Structural Lemma to prove the Main Theorem (its
second statement follows immediately from its first, so we prove the first).
Throughout, we assume the Main Theorem is false and let $G$ be a counterexample
with the fewest vertices.  Let $k=\max\{\Delta,4.2*10^{14}\}$.  We must show that
$\chi'_a(G)\le k$.  In Lemma~\ref{lem1}, we show that $G$ is 2-connected,
so we can apply the Structural Lemma to $G$.  Thus, it suffices to show that 
$G$ contains none of (C1), (C2), (C3), and (C4).
Lemma~\ref{lem1} forbids (C1), and Lemma~\ref{lem2} and Corollary~\ref{cor3}
forbid (C2).  The proofs of these results are straightforward.
For (C3) the argument is more technical, so we pull out a key piece of
it as Lemma~\ref{lem4}, before finishing the proof in Lemma~\ref{lem5}.
Finally, we handle (C4) in Lemma~\ref{lem6}, using a proof similar to that of
Lemma~\ref{lem5}.  Since the proofs for (C3) and (C4) are long, we outline
our approach just prior to Lemma~\ref{lem4}.

\begin{lem}
\label{lem1}
Let $G$ be a minimal counterexample to the Main Theorem.
Now $G$ is 2-connected and has no instance of configuration (C1).  That is,
every vertex $v$ has $\sum_{w\in N(v)}d(w)>k$.
In particular, every $5^-$-vertex has a big neighbor.
\end{lem}

\begin{proof}
Let $G$ be a minimal counterexample.
Note that $G$ is connected, since otherwise one of its components is a smaller
counterexample.  
Suppose $G$ has a cut-vertex $v$, and let $G_1, G_2,\ldots$ denote the
components of $G-v$.  For each $i$, let $H_i=G[V(G_i)\cup\{v\}]$, the subgraph
formed from $G_i$ by adding all edges between $v$ and $V(G_i)$.
By minimality, each $H_i$ has an acyclic edge-coloring, say $\vph_i$.  By
permuting colors, we can assume that the sets of colors seen by $v$ in the
distinct $\vph_i$ are disjoint.  Now identifying the copies of $v$ in each $H_i$
gives a acyclic edge-coloring of $G$, a contradiction. Thus, $G$ must be 2-connected.

Suppose that $G$ has a vertex $v$ such that $\sum_{w\in N(v)}d(w)\le k$.
By minimality, $G-v$ has a acyclic edge-coloring $\vph$.  We greedily extend $\vph$
to each edge incident to $v$.  We color these edges with distinct
colors that do not already appear on some edge incident to a vertex $w$ in
$N(v)$.  This is possible precisely because $\sum_{w\in N(v)}d(w)\le k$.  Since
each color seen by $v$ is seen by only one neighbor of $v$, the resulting
extension of $\vph$ is proper and has no 2-colored cycle containing $v$; thus,
it is acyclic.  This contradiction shows that $\sum_{w\in N(v)}d(w)>k$ for every
vertex $v$.  Finally, suppose some $5^-$-vertex $v$ contradicts the final
statement of the lemma.  Now $\sum_{w\in N(v)}d(w)\le d(v)(8680)\le 5(8680)\le
k$, a contradiction.  Thus, the lemma is true.
\end{proof}

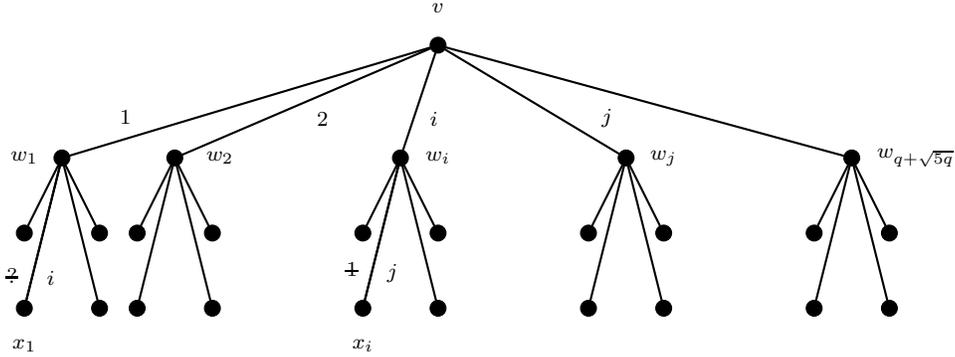
\begin{figure}[h]
\centering
\begin{tikzpicture}[scale = 10]
\tikzstyle{VertexStyle} = []
\tikzstyle{EdgeStyle} = []
\tikzstyle{labeledStyle}=[shape = circle, minimum size = 6pt, inner sep = 1.2pt]
\tikzstyle{unlabeledStyle}=[shape = circle, minimum size = 6pt, inner sep = 1.2pt, draw, fill]
\Vertex[style = unlabeledStyle, x = 0.700, y = 0.850, L = \tiny {}]{v0}
\Vertex[style = unlabeledStyle, x = 0.950, y = 0.700, L = \tiny {}]{v1}
\Vertex[style = unlabeledStyle, x = 0.900, y = 0.600, L = \tiny {}]{v2}
\Vertex[style = unlabeledStyle, x = 1.000, y = 0.600, L = \tiny {}]{v3}
\Vertex[style = unlabeledStyle, x = 0.900, y = 0.500, L = \tiny {}]{v4}
\Vertex[style = unlabeledStyle, x = 1.000, y = 0.500, L = \tiny {}]{v5}
\Vertex[style = unlabeledStyle, x = 0.200, y = 0.700, L = \tiny {}]{v6}
\Vertex[style = unlabeledStyle, x = 0.150, y = 0.600, L = \tiny {}]{v7}
\Vertex[style = unlabeledStyle, x = 0.250, y = 0.600, L = \tiny {}]{v8}
\Vertex[style = unlabeledStyle, x = 0.150, y = 0.500, L = \tiny {}]{v9}
\Vertex[style = unlabeledStyle, x = 0.250, y = 0.500, L = \tiny {}]{v10}
\Vertex[style = unlabeledStyle, x = 0.350, y = 0.700, L = \tiny {}]{v11}
\Vertex[style = unlabeledStyle, x = 0.300, y = 0.600, L = \tiny {}]{v12}
\Vertex[style = unlabeledStyle, x = 0.400, y = 0.600, L = \tiny {}]{v13}
\Vertex[style = unlabeledStyle, x = 0.300, y = 0.500, L = \tiny {}]{v14}
\Vertex[style = unlabeledStyle, x = 0.400, y = 0.500, L = \tiny {}]{v15}
\Vertex[style = unlabeledStyle, x = 0.650, y = 0.700, L = \tiny {}]{v16}
\Vertex[style = unlabeledStyle, x = 0.600, y = 0.600, L = \tiny {}]{v17}
\Vertex[style = unlabeledStyle, x = 0.700, y = 0.600, L = \tiny {}]{v18}
\Vertex[style = unlabeledStyle, x = 0.600, y = 0.500, L = \tiny {}]{v19}
\Vertex[style = unlabeledStyle, x = 0.700, y = 0.500, L = \tiny {}]{v20}
\Vertex[style = unlabeledStyle, x = 1.250, y = 0.700, L = \tiny {}]{v21}
\Vertex[style = unlabeledStyle, x = 1.200, y = 0.600, L = \tiny {}]{v22}
\Vertex[style = unlabeledStyle, x = 1.300, y = 0.600, L = \tiny {}]{v23}
\Vertex[style = unlabeledStyle, x = 1.200, y = 0.500, L = \tiny {}]{v24}
\Vertex[style = unlabeledStyle, x = 1.300, y = 0.500, L = \tiny {}]{v25}
\Vertex[style = labeledStyle, x = 0.150, y = 0.450, L = \scriptsize {$x_1$}]{v26}
\Vertex[style = labeledStyle, x = 0.600, y = 0.450, L = \scriptsize {$x_i$}]{v27}
\Vertex[style = labeledStyle, x = 1.000, y = 0.700, L = \scriptsize {$w_j$}]{v28}
\Vertex[style = labeledStyle, x = 0.700, y = 0.700, L = \scriptsize {$w_i$}]{v29}
\Vertex[style = labeledStyle, x = 0.410, y = 0.700, L = \scriptsize {$w_2$}]{v30}
\Vertex[style = labeledStyle, x = 0.150, y = 0.700, L = \scriptsize {$w_1$}]{v31}
\Vertex[style = labeledStyle, x = 1.335, y = 0.700, L = \scriptsize {$w_{q+\sqrt{5q}}$}]{v32}
\Vertex[style = labeledStyle, x = 0.700, y = 0.900, L = \scriptsize {$v$}]{v33}
\Edge[label = \tiny {}, labelstyle={auto=right, fill=none}](v2)(v1)
\Edge[label = \tiny {}, labelstyle={auto=right, fill=none}](v3)(v1)
\Edge[label = \tiny {}, labelstyle={auto=right, fill=none}](v4)(v1)
\Edge[label = \tiny {}, labelstyle={auto=right, fill=none}](v5)(v1)
\Edge[label = \tiny {}, labelstyle={auto=right, fill=none}](v7)(v6)
\Edge[label = \tiny {}, labelstyle={auto=right, fill=none}](v8)(v6)
\Edge[label = \scriptsize {\sout{?}}, labelstyle={auto=left, fill=none, pos=.05}](v9)(v6)
\Edge[label = \scriptsize {$i$}, labelstyle={auto=right, fill=none, pos=.3}](v9)(v6)
\Edge[label = \tiny {}, labelstyle={auto=right, fill=none}](v10)(v6)
\Edge[label = \tiny {}, labelstyle={auto=right, fill=none}](v12)(v11)
\Edge[label = \tiny {}, labelstyle={auto=right, fill=none}](v13)(v11)
\Edge[label = \tiny {}, labelstyle={auto=right, fill=none}](v14)(v11)
\Edge[label = \tiny {}, labelstyle={auto=right, fill=none}](v15)(v11)
\Edge[label = \tiny {}, labelstyle={auto=right, fill=none}](v17)(v16)
\Edge[label = \tiny {}, labelstyle={auto=right, fill=none}](v18)(v16)
\Edge[label = \scriptsize {\sout{1}}, labelstyle={auto=left, fill=none, pos=.1}](v19)(v16)
\Edge[label = \scriptsize {$j$}, labelstyle={auto=right, fill=none, pos=.35}](v19)(v16)
\Edge[label = \tiny {}, labelstyle={auto=right, fill=none}](v20)(v16)
\Edge[label = \scriptsize {$j$}, labelstyle={auto=right, fill=none, pos=.16}](v1)(v0)
\Edge[label = \scriptsize {$1$}, labelstyle={auto=left, fill=none, pos=.2}](v6)(v0)
\Edge[label = \scriptsize {$2$}, labelstyle={auto=right, fill=none}](v11)(v0)
\Edge[label = \scriptsize {$i$}, labelstyle={auto=right, fill=none}](v16)(v0)
\Edge[label = \tiny {}, labelstyle={auto=right, fill=none}](v22)(v21)
\Edge[label = \tiny {}, labelstyle={auto=right, fill=none}](v23)(v21)
\Edge[label = \tiny {}, labelstyle={auto=right, fill=none}](v24)(v21)
\Edge[label = \tiny {}, labelstyle={auto=right, fill=none}](v25)(v21)
\Edge[label = \tiny {}, labelstyle={auto=right, fill=none}](v21)(v0)
\end{tikzpicture}
\caption{A big vertex $v$ and its set $\W$ of $5^-$-neighbors with $v$ as their
unique big neighbor, as in Lemma~\ref{lem2}.\label{fig:lem2}}
\end{figure}

\begin{lem}
\label{lem2}
Fix an integer $q$ such that $q\ge 100$.
Now $G$ cannot have a vertex $v$ such that $\card{\W}\ge
q+\sqrt{5q}$, where $\W$ is the set of $5^-$-neighbors $w$ of $v$ such that
$\sum_{x\in N(w)\setminus v}d(x)< q$. 
\end{lem}
\begin{proof}
Suppose the lemma is false, and that $q$, $G$, and $v$ witness this.\aside{$q$,
$G$, $v$} Let \EmphE{$\W$}{4mm} be the set of these neighbors of $v$; see
Figure~\ref{fig:lem2} for an example.  Pick an arbitrary $w_1\in \W$. 
By minimality, $G-w_1$ has an acyclic edge-coloring, $\vph$.  
We can greedily extend this coloring to $G-w_1+vw_1$ (and we still call it $\vph$).
Let $w_1, w_2, \ldots$ denote the vertices of $\W$. Let \Emph{$\S$} be the set of colors
either not used on an edge incident to $v$ or else used on an edge from $v$ to a neighbor
in $\W$.  For each neighbor $w_i$, by symmetry we assume that $\vph(vw_i)=i$. 
For each $w_i$, let \Emph{$\C_i$} be the set of colors used on edges incident to
vertices in $N(w_i)\setminus v$.  For each $i$, let $\S_i=\S\setminus
(\C_i\cup\{i\})$.\aside{$\S_i$}  
Any coloring obtained from $\vph$ by coloring $x_1w_1$ with a color $i\in \S_1$
or by recoloring an edge incident to $w_i$ (and not $v$) with a color $j\in
\S_i$ is a proper edge-coloring, and we explain further how to find one that is
an acyclic edge-coloring.

Let \Emph{$x_1$} be an arbitrary neighbor in $G$ of $w_1$, other than $v$.  
We now show how to extend the acyclic edge-coloring to $w_1x_1$.  This will complete the
proof, since the same argument can be repeated to extend the coloring to each
other uncolored edge of $G$ incident to $w_1$.

If we color $w_1x_1$ with any $i\in \S_1$, then any 2-colored cycle
we create must use edges $x_1w_1$, $w_1v$, and $vw_i$.  Such a cycle is only possible if
$w_i$ sees color 1.  So we assume $w_i$ sees color 1, for every $i\in \S_1$
(otherwise we can extend the coloring to $w_1x_1$).  Now for each $i\in \S_1$,
define $x_i$ such that $\vph(w_ix_i)=1$.  (Note that $x_i=x_1$ for at most one
value of $x_i$, so we can essentially ignore this case.)

Our goal is to find indices $i$ and $j$ such that $i\in \S_1$ and $j\in \S_i$
and $w_j$ does not see color $i$.  If we find such $i$ and $j$, then we color
$w_1x_1$ with $i$ and recolor $w_ix_i$ with $j$ (as in Figure~\ref{fig:lem2}). 
This creates no 2-colored cycles, as
we now show.  Any 2-colored cycle using $w_1x_1$ must also use $w_1v$ and
$vw_i$ (since $i\in \S_1$).  But no such 2-colored cycle exists, since $w_i$ no
longer sees 1.  Similarly, any 2-colored cycle using edge $w_ix_i$ also uses 
$w_iv$ and $vw_j$.  Again, no such 2-colored cycle exists, since $w_j$ does not see
$i$.

Now we show that we can find such $i$ and $j$.  Suppose not.  So for each $i\in
\S_1$ and $j\in\S_i$ vertex $w_j$ sees color $i$.  Thus, among the at most
$4\card{\W}$ edges incident to some $w_i$, but not to $v$, each color $i\in
\S_1$ appears
at least $\card{\S_i}$ times.  Since $\card{\S_i}\ge\card{\W}-(\card{\C_i}+1)\ge
\card{\W}-q$, we get $(\card{\W}-q)^2\le 4\card{\W}$.  Solving this quadratic gives
$\card{\W}\le q+2+\sqrt{4q+4}$.  But this quantity is less than
$q+\sqrt{5q}$ when $q>80$, a contradiction.
\end{proof}

\vfill

\begin{cor}
\label{cor3}
Configuration (C2) cannot appear in a minimal counterexample $G$.  That is, $G$
has no big vertex $v$ such that among those
$5^-$-vertices with $v$ as their unique big neighbor the number of
(i) 2-vertices is at least 8889 or
(ii) $3^-$-vertices is at least 17655 or
(iii) $4^-$-vertices is at least 26401 or
(iv) $5^-$-vertices is at least 35137.
\end{cor}
\begin{proof}
This is a direct application of the previous lemma, for each $q\in\{8680, 17360,
26040, 34720\}$.  Each $5^-$-vertex $w$ with
$v$ as its only big neighbor has $\sum_{x\in d(w)\setminus v}d(x) <
(d(w)-1)8680$.  Thus, for $2$-vertices, $3$-vertices, $4$-vertices, and
$5$-vertices, the sums are (respectively) at most 8680, 17360, 26040, and
34720.  Now we are done, since $8680+\sqrt{5(8680)}\le 8889$;
$17360+\sqrt{5(17360)}\le 17655$; $26040+\sqrt{5(26040)}\le 26401$; and
$34720+\sqrt{5(34720)}\le 35137$.  
\end{proof}

\vfill

Next we turn to the task of proving that neither (C3) nor (C4) can appear in a
minimal counterexample.  Since the proofs are long, we sketch the ideas below.
The proof for (C4) is similar to that for (C3), so we just sketch the
latter.  First, we need definitions.  Recall that, for a bunch with parents
$v$ and $w$ and bunch vertices $x_0,\ldots,x_{t+1}$, the \EmphE{length}{-5mm}
is $t$ and each path $vx_iw$ is a \EmphE{thread}{0mm}.  A \emph{horizontal edge}\aaside{horizon- 
tal edge}{-1mm} is any edge $x_ix_{i+1}$, with $1\le i\le t-1$.  For a bunch $B$ in a
graph $G$, we form $G_B$ from $G$\aaside{$B$, $G_B$}{2mm} by deleting all horizontal
edges of $B$ (recall that this does not delete $x_0x_1$ and $x_tx_{t+1}$).  Now
$B$ is \EmphE{long}{1mm} if, given any integer $k\ge 13$ and any acyclic
$k$-edge-coloring of $G_B$, there exists an acyclic $k$-edge-coloring of $G$. 
If $B$ is not long, then it is \Emph{short}.
%

Using a counting argument, we show that the big vertex $v$ in (C3) has all but
a constant number, $a_0$, of its incident edges in long bunches; further, one of
these bunches, say $B_0$, has length greater than $a_0$.  We delete all horizontal
edges in long bunches, as well as one edge, $e$, incident to $v$ that lies on a
thread in $B_0$.  By minimality, the resulting graph has an acyclic
$k$-edge-coloring.  By repeatedly recoloring threads incident to $v$ in long
bunches, we can eventually extend this coloring to $e$.  Finally, we extend the
coloring to all the horizontal edges in long bunches.  
We now prove that every bunch of length at least 11 is long.  (It is easy to
check that bunches of sufficiently large length (say 22) are long, but relying
on this weaker bound would increase the value $4.2*10^{14}$ in our Main
Theorem.)  
\vfill
\newpage

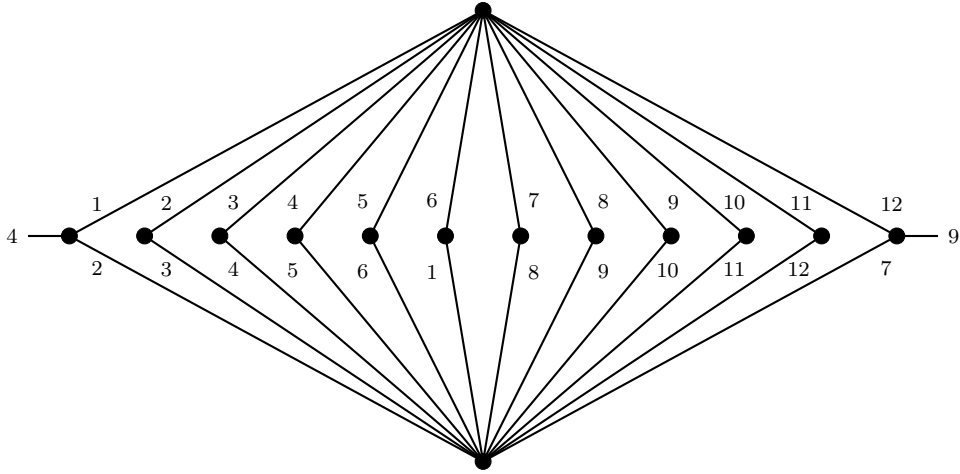
\begin{figure}[!h]
\centering
\begin{tikzpicture}[scale = 10]
\tikzstyle{VertexStyle} = []
\tikzstyle{EdgeStyle} = []
\tikzstyle{smallish}=[shape=circle, minimum size=1pt, inner sep=0pt]
\tikzstyle{labeledStyle}=[shape = circle, minimum size = 6pt, inner sep = 1.2pt, draw]
\tikzstyle{unlabeledStyle}=[shape = circle, minimum size = 6pt, inner sep = 1.2pt, draw, fill]
\Vertex[style = smallish, x=-.035, y=.650, L=\tiny {}]{v14}
\Vertex[style = smallish, x=1.235, y=.650, L=\tiny {}]{v15}
\Vertex[style = unlabeledStyle, x = 0.250, y = 0.650, L = \tiny {}]{v0}
\Vertex[style = unlabeledStyle, x = 0.350, y = 0.650, L = \tiny {}]{v1}
\Vertex[style = unlabeledStyle, x = 0.450, y = 0.650, L = \tiny {}]{v2}
\Vertex[style = unlabeledStyle, x = 0.550, y = 0.650, L = \tiny {}]{v3}
\Vertex[style = unlabeledStyle, x = 0.650, y = 0.650, L = \tiny {}]{v4}
\Vertex[style = unlabeledStyle, x = 0.150, y = 0.650, L = \tiny {}]{v5}
\Vertex[style = unlabeledStyle, x = 0.750, y = 0.650, L = \tiny {}]{v6}
\Vertex[style = unlabeledStyle, x = 0.850, y = 0.650, L = \tiny {}]{v7}
\Vertex[style = unlabeledStyle, x = 0.950, y = 0.650, L = \tiny {}]{v8}
\Vertex[style = unlabeledStyle, x = 1.050, y = 0.650, L = \tiny {}]{v9}
\Vertex[style = unlabeledStyle, x = 0.050, y = 0.650, L = \tiny {}]{v10}
\Vertex[style = unlabeledStyle, x = 1.150, y = 0.650, L = \tiny {}]{v11}
\Vertex[style = unlabeledStyle, x = 0.600, y = 0.950, L = \tiny {}]{v12}
\Vertex[style = unlabeledStyle, x = 0.600, y = 0.350, L = \tiny {}]{v13}
%
%
\Edge[label = \scriptsize {$1$}, labelstyle={auto=left, pos=.05, fill=none,
xshift=.09in}](v10)(v12)
\Edge[label = \scriptsize {$2$}, labelstyle={auto=left, pos=.05, fill=none,
xshift=.08in}](v5)(v12)
\Edge[label = \scriptsize {$3$}, labelstyle={auto=left, pos=.05, fill=none,
xshift=.06in}](v0)(v12)
\Edge[label = \scriptsize {$4$}, labelstyle={auto=left, pos=.05, fill=none}](v1)(v12)
\Edge[label = \scriptsize {$5$}, labelstyle={auto=left, pos=.05, fill=none}](v2)(v12)
\Edge[label = \scriptsize {$6$}, labelstyle={auto=left, pos=.05, fill=none}](v3)(v12)
\Edge[label = \scriptsize {$7$}, labelstyle={auto=right, pos=.05, fill=none}](v4)(v12)
\Edge[label = \scriptsize {$8$}, labelstyle={auto=right, pos=.05, fill=none}](v6)(v12)
\Edge[label = \scriptsize {$9$}, labelstyle={auto=right, pos=.05, fill=none}](v7)(v12)
\Edge[label = \scriptsize {$10$}, labelstyle={auto=right, pos=.05, fill=none,
xshift=-.08in}](v8)(v12)
\Edge[label = \scriptsize {$11$}, labelstyle={auto=right, pos=.05, fill=none,
xshift=-.1in}](v9)(v12)
\Edge[label = \scriptsize {$12$}, labelstyle={auto=right, pos=.05, fill=none}](v11)(v12)
%
%
\Edge[labelstyle={pos=.9}, label = \scriptsize{4}](v10)(v14)
\Edge[labelstyle={pos=.9}, label = \scriptsize{9}](v11)(v15)
%
%
\Edge[label = \scriptsize {$2$}, labelstyle={auto=right, pos=.05, fill=none,
xshift=.09in}](v10)(v13)
\Edge[label = \scriptsize {$3$}, labelstyle={auto=right, pos=.05, fill=none,
xshift=.08in}](v5)(v13)
\Edge[label = \scriptsize {$4$}, labelstyle={auto=right, pos=.05, fill=none,
xshift=.06in}](v0)(v13)
\Edge[label = \scriptsize {$5$}, labelstyle={auto=right, pos=.05, fill=none}](v1)(v13)
\Edge[label = \scriptsize {$6$}, labelstyle={auto=right, pos=.05, fill=none}](v2)(v13)
\Edge[label = \scriptsize {$1$}, labelstyle={auto=right, pos=.05, fill=none}](v3)(v13)
\Edge[label = \scriptsize {$8$}, labelstyle={auto=left, pos=.05, fill=none}](v4)(v13)
\Edge[label = \scriptsize {$9$}, labelstyle={auto=left, pos=.05, fill=none}](v6)(v13)
\Edge[label = \scriptsize {$10$}, labelstyle={auto=left, pos=.05, fill=none,
xshift=-.06in}](v7)(v13)
\Edge[label = \scriptsize {$11$}, labelstyle={auto=left, pos=.05, fill=none,
xshift=-.08in}](v8)(v13)
\Edge[label = \scriptsize {$12$}, labelstyle={auto=left, pos=.05, fill=none,
xshift=-.115in}](v9)(v13)
\Edge[label = \scriptsize {$7$}, labelstyle={auto=left, pos=.05, fill=none}](v11)(v13)
\end{tikzpicture}
\caption{An acyclic edge-coloring of $G_B$, restricted to the edges incident to bunch
vertices of $B$, where $B$ is a bunch of length 12.\label{fig1}
}
\end{figure}
\bigskip
\bigskip

\begin{lem}
\label{lem4}
In every planar graph, every bunch of length at least 11 is long.
\end{lem}

\begin{proof}
Consider a planar graph $G$ with a bunch, \Emph{$B$}, of length at least 11. 
Fix an integer $k\ge 13$.  We may assume $G_B$ has an acyclic
$k$-edge-coloring; see Figure~\ref{fig1} for an example.  Let $v$ and
$w$\aside{$v$, $w$} be the parents of the bunch and let $x_1,\ldots,x_t$ denote
its vertices.  We will show that we can reorder the threads of $B$ so that (for
each $i\in\irange{t-1}$) no color appears incident to both $x_i$ and $x_{i+1}$.
 (Technically, we reorder the pairs of colors on the edges $vx_i$ and
$x_iw$, while preserving, in each pair, which color is incident to $v$ and
which is incident to $w$; but this minor distinction will not trouble us.)
We also require that the colors seen by $x_2$ do not appear on edge $x_0x_1$ and,
similarly, the colors seen by $x_{t-1}$ do not appear on edge $x_tx_{t+1}$.  If we
can reorder the threads to achieve this property, then it is easy to extend the
$k$-edge-coloring to $G$, as follows.

We greedily color the horizontal edges in any order, requiring that the
color used on $x_ix_{i+1}$ not appear on any (colored) edge incident to
$x_{i-1}$, $x_i$, $x_{i+1}$, or $x_{i+2}$.  Each of these vertices has two
incident edges on a thread, for a total of 8 edges.  We must also avoid the
colors on at most 4 horizontal edges.  Thus, at most 12 colors are forbidden. 
Since $k\ge 13$, we greedily complete the coloring.  Given an acyclic
$k$-edge-coloring of $G_B$, suppose that we reorder the threads of $G_B$ and
greedily extend the coloring to the horizontal edges of $B$.  Call the
resulting $k$-edge-coloring \Emph{$\vph$}.  Clearly, $\vph$ is a proper edge-coloring.
We must also show that it has no 2-colored cycles.  Suppose, to the contrary,
that $\vph$ has a 2-colored cycle, $C$.  By the condition on our ordering of the
threads of $B$, the cycle $C$ must use at least two successive horizontal
edges of $B$.  But now one of these horizontal edges $x_ix_{i+1}$ of $C$ must
share a color with an edge incident to $x_{i-1}$ or $x_{i+2}$, a contradiction.
 Thus, $\vph$ is an acyclic $k$-edge-coloring of $G$, as desired.  Hence, it
suffices to show that we can reorder the threads of $B$ so that no
color appears incident to both $x_i$ and $x_{i+1}$.

For each $i\in\irange{t}$, we think of putting some thread $vx_jw$ into
position $i$ (where also $j\in\irange{t}$).  We always put thread 1 into
position 1 and thread $t$ into position $t$.  We will also initially put
threads into the positions with $i$ odd.  Let \Emph{$\O$} be the set of threads that we
put in the odd positions (and thread $t$, whether or not $t$ is odd); $\O$ is
for odd.  Note that $\card{\O}=\ceil{(t+1)/2}$.  Later, we put threads into
the even positions.  To do so, after putting threads into the odd positions, we
build a bipartite graph, \Emph{$H(B,\O)$}, where the vertices of one part are
the even numbered positions (excluding $t$) and the vertices of the other part
are those threads not yet placed.  We add an edge between a thread $vx_iw$ and
a position $j$ if no color used on the thread is also used on a thread already
in position $j-1$ or $j+1$, or used on $x_0x_1$ when $j=2$, or on
$x_tx_{t+1}$ when $j=t-1$; see Figure~\ref{fig4} for an example.  (The notation
$H(B,\O)$ is slightly misleading,
since the edges of this graph depend not only on our choice of $\O$, but also on
which threads we put where.) Thus, to place the remaining threads, it suffices
to find a perfect matching in $H(B,\O)$.  When $t\ge 22$, we can put threads
into the odd positions essentially arbitrarily, and we are guaranteed a perfect
matching in $H(B,\O)$ by a straightforward application of Hall's
Theorem. This approach allows us to complete the proof, but requires that we
replace $4.2*10^{14}$ with a larger constant.  For smaller $t$, we use a similar
approach, but need more detailed case analysis.

We build a \Emph{conflict graph}, \EmphE{$\Bconf$}{8mm}, which has as its vertices the threads
of $B$, that is, $vx_iw$, for all $i\in\irange{t}$.  Two vertices are adjacent
in $\Bconf$ if their corresponding threads share a common color; see
Figure~\ref{fig2} for an example.  Note that $\Bconf$ is a disjoint union of
paths and cycles, since every edge in a thread of $B$ is incident to either $v$
or $w$.  We refer interchangeably to a thread and its corresponding vertex in
$\Bconf$.  To form $\O$ we start with an empty set and repeatedly add vertices,
subject to the following condition.  Each component of $\Bconf$ with a vertex
in $\O$ must have all of its vertices in $\O$, except for at most one
component; if such a component exists, then its vertices that are in $\O$ must
induce a path.  Thus, at most two threads in $\O$ have neighbors in $\Bconf$
that are not in $\O$ (and if exactly two, then each has at most one such neighbor).  

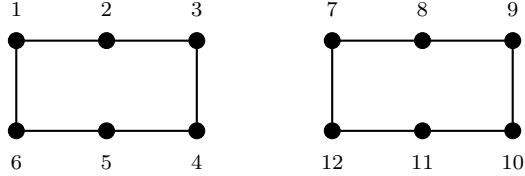
\begin{figure}[t]
\centering
\begin{tikzpicture}[scale = 12]
\tikzstyle{VertexStyle} = []
\tikzstyle{EdgeStyle} = []
\tikzstyle{labeledStyle}=[shape = circle, minimum size = 6pt, inner sep = 1.2pt ]
\tikzstyle{unlabeledStyle}=[shape = circle, minimum size = 6pt, inner sep = 1.2pt, draw, fill]
\Vertex[style = unlabeledStyle, x = 0.800, y = 0.850, L = \tiny {}]{v0}
\Vertex[style = unlabeledStyle, x = 0.900, y = 0.850, L = \tiny {}]{v1}
\Vertex[style = unlabeledStyle, x = 1.000, y = 0.850, L = \tiny {}]{v2}
\Vertex[style = unlabeledStyle, x = 1.000, y = 0.750, L = \tiny {}]{v3}
\Vertex[style = unlabeledStyle, x = 0.900, y = 0.750, L = \tiny {}]{v4}
\Vertex[style = unlabeledStyle, x = 0.800, y = 0.750, L = \tiny {}]{v5}
\Vertex[style = unlabeledStyle, x = 0.450, y = 0.850, L = \tiny {}]{v6}
\Vertex[style = unlabeledStyle, x = 0.550, y = 0.850, L = \tiny {}]{v7}
\Vertex[style = unlabeledStyle, x = 0.650, y = 0.850, L = \tiny {}]{v8}
\Vertex[style = unlabeledStyle, x = 0.650, y = 0.750, L = \tiny {}]{v9}
\Vertex[style = unlabeledStyle, x = 0.550, y = 0.750, L = \tiny {}]{v10}
\Vertex[style = unlabeledStyle, x = 0.450, y = 0.750, L = \tiny {}]{v11}
\Vertex[style = labeledStyle, x = 0.450, y = 0.885, L = \scriptsize {$1$}]{v12}
\Vertex[style = labeledStyle, x = 0.550, y = 0.885, L = \scriptsize {$2$}]{v13}
\Vertex[style = labeledStyle, x = 0.650, y = 0.885, L = \scriptsize {$3$}]{v14}
\Vertex[style = labeledStyle, x = 0.650, y = 0.715, L = \scriptsize {$4$}]{v15}
\Vertex[style = labeledStyle, x = 0.550, y = 0.715, L = \scriptsize {$5$}]{v16}
\Vertex[style = labeledStyle, x = 0.450, y = 0.715, L = \scriptsize {$6$}]{v17}
\Vertex[style = labeledStyle, x = 0.800, y = 0.885, L = \scriptsize {$7$}]{v18}
\Vertex[style = labeledStyle, x = 0.900, y = 0.885, L = \scriptsize {$8$}]{v19}
\Vertex[style = labeledStyle, x = 1.000, y = 0.885, L = \scriptsize {$9$}]{v20}
\Vertex[style = labeledStyle, x = 1.000, y = 0.715, L = \scriptsize {$10$}]{v21}
\Vertex[style = labeledStyle, x = 0.900, y = 0.715, L = \scriptsize {$11$}]{v22}
\Vertex[style = labeledStyle, x = 0.800, y = 0.715, L = \scriptsize {$12$}]{v23}
\Edge[label = \tiny {}, labelstyle={auto=right, fill=none}](v1)(v0)
\Edge[label = \tiny {}, labelstyle={auto=right, fill=none}](v5)(v0)
\Edge[label = \tiny {}, labelstyle={auto=right, fill=none}](v1)(v2)
\Edge[label = \tiny {}, labelstyle={auto=right, fill=none}](v3)(v2)
\Edge[label = \tiny {}, labelstyle={auto=right, fill=none}](v3)(v4)
\Edge[label = \tiny {}, labelstyle={auto=right, fill=none}](v5)(v4)
\Edge[label = \tiny {}, labelstyle={auto=right, fill=none}](v7)(v6)
\Edge[label = \tiny {}, labelstyle={auto=right, fill=none}](v7)(v8)
\Edge[label = \tiny {}, labelstyle={auto=right, fill=none}](v9)(v8)
\Edge[label = \tiny {}, labelstyle={auto=right, fill=none}](v9)(v10)
\Edge[label = \tiny {}, labelstyle={auto=right, fill=none}](v11)(v6)
\Edge[label = \tiny {}, labelstyle={auto=right, fill=none}](v11)(v10)
\end{tikzpicture}
\caption{The conflict graph, $\Bconf$. We label each vertex with the color the
thread uses on the edge to its parent ``above''.  Applying our algorithm to
this instance of $\Bconf$ yields $\O=\{1,7,8,9,10,11,12\}$.
\label{fig2}}
\end{figure}

First suppose that threads 1 and $t$ are in different components of $\Bconf$.
We begin by putting into $\O$ all threads in the smaller of these components,
and then proceed to the other component, beginning with the thread in $\{1,t\}$. 
If threads 1 and $t$ are in the same component of $\Bconf$, then we start by
putting into $\O$ all vertices on a shortest path in $\Bconf$ from $1$ to $t$,
and thereafter continue growing arbitrarily, such that when the set reaches size
$\ceil{(t+1)/2}$ it satisfies the desired property.  

%
\begin{figure}[!bh]
\centering
\begin{tikzpicture}[scale = 10]
\tikzstyle{VertexStyle} = []
\tikzstyle{EdgeStyle} = []
\tikzstyle{smallish}=[shape=circle, minimum size=1pt, inner sep=0pt]
\tikzstyle{labeledStyle}=[shape = circle, minimum size = 6pt, inner sep = 1.2pt, draw]
\tikzstyle{unlabeledStyle}=[shape = circle, minimum size = 6pt, inner sep = 1.2pt, draw, fill]
\Vertex[style = smallish, x=-.035, y=.650, L=\tiny {}]{v14}
\Vertex[style = smallish, x=1.235, y=.650, L=\tiny {}]{v15}
\Vertex[style = unlabeledStyle, x = 0.250, y = 0.650, L = \tiny {}]{v0}
\Vertex[style = unlabeledStyle, x = 0.350, y = 0.650, L = \tiny {}]{v1}
\Vertex[style = unlabeledStyle, x = 0.450, y = 0.650, L = \tiny {}]{v2}
\Vertex[style = unlabeledStyle, x = 0.550, y = 0.650, L = \tiny {}]{v3}
\Vertex[style = unlabeledStyle, x = 0.650, y = 0.650, L = \tiny {}]{v4}
\Vertex[style = unlabeledStyle, x = 0.150, y = 0.650, L = \tiny {}]{v5}
\Vertex[style = unlabeledStyle, x = 0.750, y = 0.650, L = \tiny {}]{v6}
\Vertex[style = unlabeledStyle, x = 0.850, y = 0.650, L = \tiny {}]{v7}
\Vertex[style = unlabeledStyle, x = 0.950, y = 0.650, L = \tiny {}]{v8}
\Vertex[style = unlabeledStyle, x = 1.050, y = 0.650, L = \tiny {}]{v9}
\Vertex[style = unlabeledStyle, x = 0.050, y = 0.650, L = \tiny {}]{v10}
\Vertex[style = unlabeledStyle, x = 1.150, y = 0.650, L = \tiny {}]{v11}
\Vertex[style = unlabeledStyle, x = 0.600, y = 0.950, L = \tiny {}]{v12}
\Vertex[style = unlabeledStyle, x = 0.600, y = 0.350, L = \tiny {}]{v13}
%
%
\Edge[label = \scriptsize {$1$}, labelstyle={auto=left, pos=.05, fill=none,
xshift=.09in}](v10)(v12)
\Edge[label = \scriptsize {}, labelstyle={auto=left, pos=.05, fill=none,
xshift=.08in}](v5)(v12)
\Edge[label = \scriptsize {$7$}, labelstyle={auto=left, pos=.05, fill=none,
xshift=.06in}](v0)(v12)
\Edge[label = \scriptsize {}, labelstyle={auto=left, pos=.05, fill=none}](v1)(v12)
\Edge[label = \scriptsize {$8$}, labelstyle={auto=left, pos=.05, fill=none}](v2)(v12)
\Edge[label = \scriptsize {}, labelstyle={auto=left, pos=.05, fill=none}](v3)(v12)
\Edge[label = \scriptsize {$11$}, labelstyle={auto=right, pos=.05, fill=none}](v4)(v12)
\Edge[label = \scriptsize {}, labelstyle={auto=right, pos=.05, fill=none}](v6)(v12)
\Edge[label = \scriptsize {$9$}, labelstyle={auto=right, pos=.05, fill=none}](v7)(v12)
\Edge[label = \scriptsize {}, labelstyle={auto=right, pos=.05, fill=none,
xshift=-.08in}](v8)(v12)
\Edge[label = \scriptsize {$10$}, labelstyle={auto=right, pos=.05, fill=none,
xshift=-.1in}](v9)(v12)
\Edge[label = \scriptsize {$12$}, labelstyle={auto=right, pos=.05, fill=none}](v11)(v12)
%
%
\Edge[labelstyle={pos=.9}, label = \scriptsize{4}](v10)(v14)
\Edge[labelstyle={pos=.9}, label = \scriptsize{9}](v11)(v15)
%
%
\Edge[label = \scriptsize {$2$}, labelstyle={auto=right, pos=.05, fill=none,
xshift=.09in}](v10)(v13)
\Edge[label = \scriptsize {}, labelstyle={auto=right, pos=.05, fill=none,
xshift=.08in}](v5)(v13)
\Edge[label = \scriptsize {$8$}, labelstyle={auto=right, pos=.05, fill=none,
xshift=.06in}](v0)(v13)
\Edge[label = \scriptsize {}, labelstyle={auto=right, pos=.05, fill=none}](v1)(v13)
\Edge[label = \scriptsize {$9$}, labelstyle={auto=right, pos=.05, fill=none}](v2)(v13)
\Edge[label = \scriptsize {}, labelstyle={auto=right, pos=.05, fill=none}](v3)(v13)
\Edge[label = \scriptsize {$12$}, labelstyle={auto=left, pos=.05, fill=none}](v4)(v13)
\Edge[label = \scriptsize {}, labelstyle={auto=left, pos=.05, fill=none}](v6)(v13)
\Edge[label = \scriptsize {$10$}, labelstyle={auto=left, pos=.05, fill=none}](v7)(v13)
\Edge[label = \scriptsize {}, labelstyle={auto=left, pos=.05, fill=none,
xshift=-.08in}](v8)(v13)
\Edge[label = \scriptsize {$11$}, labelstyle={auto=left, pos=.05, fill=none,
xshift=-.1in}](v9)(v13)
\Edge[label = \scriptsize {$7$}, labelstyle={auto=left, pos=.05, fill=none}](v11)(v13)
\end{tikzpicture}
\caption{A partial acyclic edge-coloring of $G_B$, with the threads in $\O$ in
position.\label{fig3}}
\end{figure}
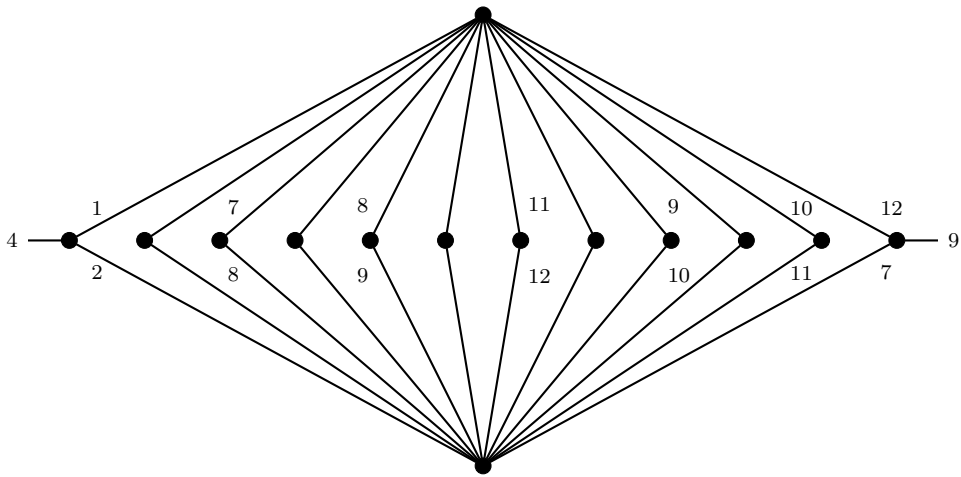

The only exception is if
the shortest path from 1 to $t$ has more than $\ceil{(t+1)/2}$ vertices.  In
this case the component of $\Bconf$ is a path; now we add a single edge in
$\Bconf$ joining its endpoints, and proceed as above, which allows us to take
a shorter path from 1 to $t$, including the edge we just added.  Thus, we have
constructed the desired $\O$.

%
%
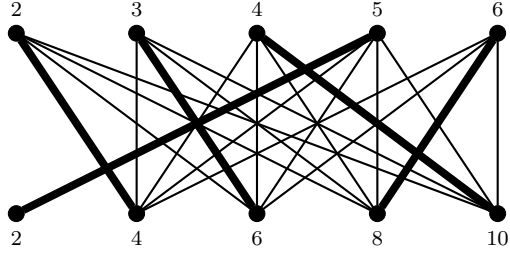
\begin{figure}
\centering
\begin{tikzpicture}[scale = 16]
\tikzset{
    uber thick/.style= {line width=3pt},
}
\tikzstyle{VertexStyle} = []
\tikzstyle{EdgeStyle} = []
\tikzstyle{labeledStyle}=[shape = circle, minimum size = 6pt, inner sep = 1.2pt ]
\tikzstyle{unlabeledStyle}=[shape = circle, minimum size = 6pt, inner sep = 1.2pt, draw, fill]
\Vertex[style = unlabeledStyle, x = 0.500, y = 0.850, L = \tiny {}]{v0}
\Vertex[style = unlabeledStyle, x = 0.600, y = 0.850, L = \tiny {}]{v1}
\Vertex[style = unlabeledStyle, x = 0.700, y = 0.850, L = \tiny {}]{v2}
\Vertex[style = unlabeledStyle, x = 0.800, y = 0.850, L = \tiny {}]{v3}
\Vertex[style = unlabeledStyle, x = 0.900, y = 0.850, L = \tiny {}]{v4}
\Vertex[style = unlabeledStyle, x = 0.900, y = 0.700, L = \tiny {}]{v5}
\Vertex[style = unlabeledStyle, x = 0.800, y = 0.700, L = \tiny {}]{v6}
\Vertex[style = unlabeledStyle, x = 0.700, y = 0.700, L = \tiny {}]{v7}
\Vertex[style = unlabeledStyle, x = 0.600, y = 0.700, L = \tiny {}]{v8}
\Vertex[style = unlabeledStyle, x = 0.500, y = 0.700, L = \tiny {}]{v9}
\Vertex[style = labeledStyle, x = 0.500, y = 0.870, L = \scriptsize {$2$}]{v10}
\Vertex[style = labeledStyle, x = 0.600, y = 0.870, L = \scriptsize {$3$}]{v11}
\Vertex[style = labeledStyle, x = 0.700, y = 0.870, L = \scriptsize {$4$}]{v12}
\Vertex[style = labeledStyle, x = 0.800, y = 0.870, L = \scriptsize {$5$}]{v13}
\Vertex[style = labeledStyle, x = 0.900, y = 0.870, L = \scriptsize {$6$}]{v14}
\Vertex[style = labeledStyle, x = 0.500, y = 0.680, L = \scriptsize {$2$}]{v15}
\Vertex[style = labeledStyle, x = 0.600, y = 0.680, L = \scriptsize {$4$}]{v16}
\Vertex[style = labeledStyle, x = 0.700, y = 0.680, L = \scriptsize {$6$}]{v17}
\Vertex[style = labeledStyle, x = 0.800, y = 0.680, L = \scriptsize {$8$}]{v18}
\Vertex[style = labeledStyle, x = 0.900, y = 0.680, L = \scriptsize {$10$}]{v19}
\Edge[label = \tiny {}, labelstyle={auto=right, fill=none}](v0)(v5)
\Edge[label = \tiny {}, labelstyle={auto=right, fill=none}](v1)(v5)
\Edge[label = \tiny {}, style={uber thick}, labelstyle={auto=right, fill=none}](v2)(v5)
\Edge[label = \tiny {}, labelstyle={auto=right, fill=none}](v4)(v5)
\Edge[label = \tiny {}, labelstyle={auto=right, fill=none}](v0)(v6)
\Edge[label = \tiny {}, labelstyle={auto=right, fill=none}](v1)(v6)
\Edge[label = \tiny {}, labelstyle={auto=right, fill=none}](v2)(v6)
\Edge[label = \tiny {}, style={uber thick}, labelstyle={auto=right, fill=none}](v4)(v6)
\Edge[label = \tiny {}, labelstyle={auto=right, fill=none}](v0)(v7)
\Edge[label = \tiny {}, style={uber thick}, labelstyle={auto=right, fill=none}](v1)(v7)
\Edge[label = \tiny {}, labelstyle={auto=right, fill=none}](v2)(v7)
\Edge[label = \tiny {}, labelstyle={auto=right, fill=none}](v4)(v7)
\Edge[label = \tiny {}, style={uber thick}, labelstyle={auto=right, fill=none}](v0)(v8)
\Edge[label = \tiny {}, labelstyle={auto=right, fill=none}](v1)(v8)
\Edge[label = \tiny {}, labelstyle={auto=right, fill=none}](v2)(v8)
\Edge[label = \tiny {}, labelstyle={auto=right, fill=none}](v4)(v8)
\Edge[label = \tiny {}, labelstyle={auto=right, fill=none}](v5)(v3)
\Edge[label = \tiny {}, labelstyle={auto=right, fill=none}](v6)(v3)
\Edge[label = \tiny {}, labelstyle={auto=right, fill=none}](v7)(v3)
\Edge[label = \tiny {}, labelstyle={auto=right, fill=none}](v8)(v3)
\Edge[label = \tiny {}, style={uber thick}, labelstyle={auto=right, fill=none}](v9)(v3)
\end{tikzpicture}
\caption{The auxiliary graph $H(B,\O)$, with threads on top and positions on
bottom, and a perfect matching shown in bold.\label{fig4}}
\end{figure}

First, we place the threads of $\O$ in the odd positions; second, we place the
remaining threads in the even positions, using Hall's Theorem.
See Figure~\ref{fig3} for an example of these threads in position, and
Figure~\ref{fig4} for the resulting graph $H(B,\O)$.
Let $r$ denote the size of each part in $H(B,\O)$.  Since
$\card{\O}=\ceil{(t+1)/2}$ and $t\ge 11$, we get that $r=\floor{(t-1)/2}\ge 5$.
Recall that at most two threads in $\O$ have neighbors in $\Bconf$ that are not in
$\O$ (and if exactly two, then each has at most one such neighbor).  
We consider five cases, depending on which threads in $\O$ have neighbors in
$B_{conf}$ that are not in $\O$.
\newpage

\textbf{Case 1:} 
Suppose that $1$ and $t$ are the two threads in $\O$ with neighbors in $\Bconf$
that are not in $\O$.  We put threads 1 and $t$ in their positions and we put
the other threads of $\O$ in the odd positions arbitrarily, except that if $t$
is even, then we pick a thread for position $t-1$ that does not conflict with
thread $t$ and does not conflict with the color on $x_tx_{t+1}$ (if it exists);
this is easy, since $t\ge 11$.
Now we must put the remaining threads into the even positions.  At most three
threads are forbidden from position 2, since at most one thread has a color used
on thread 1 and at most two threads have colors used on $x_0x_1$.  Similarly, at
most three threads are forbidden from position $t-1$.  For all other positions,
no threads are forbidden.  
Positions 2 and $t-1$ have
degree at least $r-3\ge 2$ in $H(B,\O)$ and all other positions have degree $r$. 
Thus, by Hall's Theorem, $H(B,\O)$ has a perfect matching.  We now use similar
arguments to handle the other possibilities for which vertices of $\O$ have
neighbors in $\Bconf$ that are not in $\O$.

\textbf{Case 2:} 
Suppose that exactly one of threads $1$ and $t$ has a neighbor in $\Bconf$ that
is not in $\O$.  By symmetry, assume that it is 1.  Further, assume that also
$i\in \O$ and thread $i$ has a neighbor in $\Bconf$ that is not in $\O$ (the
case when no such $i$ exists is easier).  If $t$ is odd, then we put thread $i$
in position $t-2$,
and fill the remaining odd positions arbitrarily from $\O$.  If $t$ is even,
then we put thread $i$ in position $t-2$, and fill odd positions 3 through $t-3$
arbitrarily from $\O$, except that we require that the thread in position $t-3$
not conflict with that in position $t-2$; this is possible, since at most
two threads in $\O$ conflict with thread $i$, and $\card{\O}\ge 7$.  
Note that here we put an element of $\O$ in position $t-2$, but not in position
$t$.
Again, we use
Hall's Theorem to show that $H(B,\O)$ has a perfect matching.  Now positions 2 and
$t-1$ each have degree at least $r-3\ge 2$, and position $t-3$ has degree at least
$r-1\ge 4$.  All other positions have degree $r$.

\textbf{Case 3:} 
Suppose that one of threads $1$ and $t$ has two neighbors in $\Bconf$ that are
not in $\O$, and the other has no such neighbors.  (This will happen when
$\Bconf$ consists of two cycles, each of length $t/2$.)  By symmetry, assume
that thread 1 has two neighbors in $\Bconf$ that are not in $\O$.  We fill the
odd positions arbitrarily with threads from $\O$ (here, and in the remaining
cases, if $t$ is even, then we also require that the thread in position $t-1$
not conflict with thread $t$ or with the color on $x_tx_{t+1}$).  In $H(B,\O)$,
position 2 has degree at least $r-4\ge 1$.  Also, position $t-1$ has degree at
least $r-2\ge 3$.  All other positions have degree $r$.  So $H(B,\O)$ has a
perfect matching.

%
\begin{figure}[!t]
\centering
\begin{tikzpicture}[scale = 10]
\tikzstyle{VertexStyle} = []
\tikzstyle{EdgeStyle} = []
\tikzstyle{smallish}=[shape=circle, minimum size=1pt, inner sep=0pt]
\tikzstyle{labeledStyle}=[shape = circle, minimum size = 6pt, inner sep = 1.2pt, draw]
\tikzstyle{unlabeledStyle}=[shape = circle, minimum size = 6pt, inner sep = 1.2pt, draw, fill]
\Vertex[style = smallish, x=-.035, y=.650, L=\tiny {}]{v14}
\Vertex[style = smallish, x=1.235, y=.650, L=\tiny {}]{v15}
\Vertex[style = unlabeledStyle, x = 0.250, y = 0.650, L = \tiny {}]{v0}
\Vertex[style = unlabeledStyle, x = 0.350, y = 0.650, L = \tiny {}]{v1}
\Vertex[style = unlabeledStyle, x = 0.450, y = 0.650, L = \tiny {}]{v2}
\Vertex[style = unlabeledStyle, x = 0.550, y = 0.650, L = \tiny {}]{v3}
\Vertex[style = unlabeledStyle, x = 0.650, y = 0.650, L = \tiny {}]{v4}
\Vertex[style = unlabeledStyle, x = 0.150, y = 0.650, L = \tiny {}]{v5}
\Vertex[style = unlabeledStyle, x = 0.750, y = 0.650, L = \tiny {}]{v6}
\Vertex[style = unlabeledStyle, x = 0.850, y = 0.650, L = \tiny {}]{v7}
\Vertex[style = unlabeledStyle, x = 0.950, y = 0.650, L = \tiny {}]{v8}
\Vertex[style = unlabeledStyle, x = 1.050, y = 0.650, L = \tiny {}]{v9}
\Vertex[style = unlabeledStyle, x = 0.050, y = 0.650, L = \tiny {}]{v10}
\Vertex[style = unlabeledStyle, x = 1.150, y = 0.650, L = \tiny {}]{v11}
\Vertex[style = unlabeledStyle, x = 0.600, y = 0.950, L = \tiny {}]{v12}
\Vertex[style = unlabeledStyle, x = 0.600, y = 0.350, L = \tiny {}]{v13}
%
%
\Edge[label = \scriptsize {$1$}, labelstyle={auto=left, pos=.05, fill=none,
xshift=.09in}](v10)(v12)
\Edge[label = \scriptsize {$5$}, labelstyle={auto=left, pos=.05, fill=none,
xshift=.08in}](v5)(v12)
\Edge[label = \scriptsize {$7$}, labelstyle={auto=left, pos=.05, fill=none,
xshift=.06in}](v0)(v12)
\Edge[label = \scriptsize {$2$}, labelstyle={auto=left, pos=.05, fill=none}](v1)(v12)
\Edge[label = \scriptsize {$8$}, labelstyle={auto=left, pos=.05, fill=none}](v2)(v12)
\Edge[label = \scriptsize {$3$}, labelstyle={auto=left, pos=.05, fill=none}](v3)(v12)
\Edge[label = \scriptsize {$11$}, labelstyle={auto=right, pos=.05, fill=none}](v4)(v12)
\Edge[label = \scriptsize {$6$}, labelstyle={auto=right, pos=.05, fill=none}](v6)(v12)
\Edge[label = \scriptsize {$9$}, labelstyle={auto=right, pos=.05, fill=none}](v7)(v12)
\Edge[label = \scriptsize {$4$}, labelstyle={auto=right, pos=.05, fill=none,
xshift=-.08in}](v8)(v12)
\Edge[label = \scriptsize {$10$}, labelstyle={auto=right, pos=.05, fill=none,
xshift=-.1in}](v9)(v12)
\Edge[label = \scriptsize {$12$}, labelstyle={auto=right, pos=.05, fill=none}](v11)(v12)
%
%
\Edge[labelstyle={pos=.9}, label = \scriptsize{4}](v10)(v14)
\Edge[labelstyle={pos=.9}, label = \scriptsize{9}](v11)(v15)
%
%
\Edge[label = \scriptsize {$2$}, labelstyle={auto=right, pos=.05, fill=none,
xshift=.09in}](v10)(v13)
\Edge[label = \scriptsize {$6$}, labelstyle={auto=right, pos=.05, fill=none,
xshift=.08in}](v5)(v13)
\Edge[label = \scriptsize {$8$}, labelstyle={auto=right, pos=.05, fill=none,
xshift=.06in}](v0)(v13)
\Edge[label = \scriptsize {$3$}, labelstyle={auto=right, pos=.05, fill=none}](v1)(v13)
\Edge[label = \scriptsize {$9$}, labelstyle={auto=right, pos=.05, fill=none}](v2)(v13)
\Edge[label = \scriptsize {$4$}, labelstyle={auto=right, pos=.05, fill=none}](v3)(v13)
\Edge[label = \scriptsize {$12$}, labelstyle={auto=left, pos=.05, fill=none}](v4)(v13)
\Edge[label = \scriptsize {$1$}, labelstyle={auto=left, pos=.05, fill=none}](v6)(v13)
\Edge[label = \scriptsize {$10$}, labelstyle={auto=left, pos=.05, fill=none}](v7)(v13)
\Edge[label = \scriptsize {$5$}, labelstyle={auto=left, pos=.05, fill=none,
xshift=-.08in}](v8)(v13)
\Edge[label = \scriptsize {$11$}, labelstyle={auto=left, pos=.05, fill=none,
xshift=-.1in}](v9)(v13)
\Edge[label = \scriptsize {$7$}, labelstyle={auto=left, pos=.05, fill=none}](v11)(v13)
\end{tikzpicture}
\caption{The desired acyclic edge-coloring of $G_B$, ready to be extended
greedily to the horizontal edges of $B$.\label{fig5}}
\end{figure}
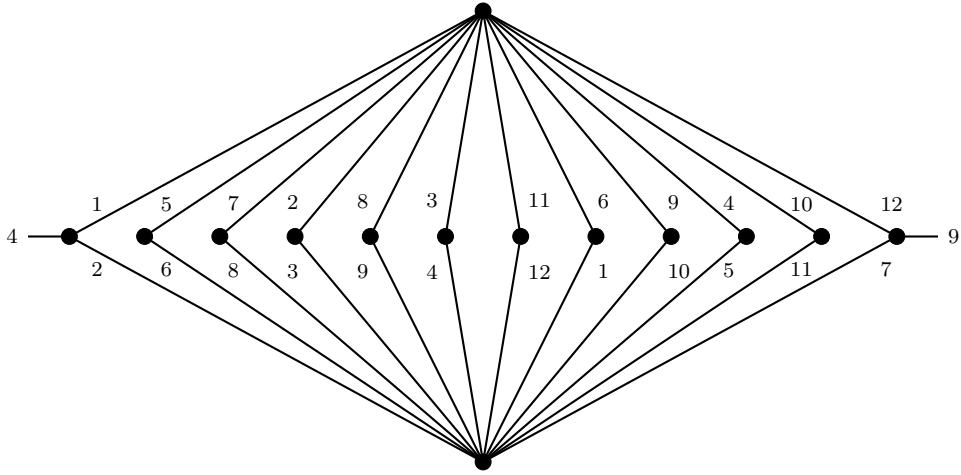

Now we can assume that neither of threads 1 and $t$ has neighbors in $\Bconf$
that are not in $\O$.

\textbf{Case 4:} 
Suppose that some thread, say $i$, in $\O$ has two neighbors in $\Bconf$ that
are not in $\O$.  We put thread $i$ in position 3 and fill the remaining
odd positions arbitrarily from $\O$.
Position $2$ has degree at least $r-4\ge 1$, and position 4 has degree at
least $r-2\ge 3$.  If $t$ is odd, then position $t-1$ has degree at least
$r-2\ge 3$.  All other positions have degree $r$.  So $H(B,\O)$ has a perfect
matching.

\textbf{Case 5:} 
Finally, suppose that two threads, $i$ and $j$ (neither of which is 1 or $t$),
each have a neighbor in $\Bconf$ that is not in $\O$.  Now we put thread $i$ in
position 3 and thread $j$ in position 5, and fill the remaining odd positions
from the rest of $\O$.
Between them, threads $i$ and $j$ forbid at most two threads from position 4
and at most one thread each from positions 2 and 6.  Thus, position 2 has degree
at least $r-3\ge 2$, position 4 has degree at least $r-2\ge 3$, and position 6
has degree at least $r-1\ge 4$.  Once again $H(B,\O)$ has a perfect matching.
%
\end{proof}

\begin{lem}
\label{lem5}
Configuration (C3) cannot appear in a minimal counterexample $G$.
That is, $G$ cannot contain a big vertex $v$ such that $\nf+2\Mns\le 35$, where
$\nf$ and $\Mns$\aside{$\nf$, $\Mns$} denote the numbers of $5^-$-neighbors and
$6^+$-neighbors of $v$ that are in no bunch with $v$ as a parent.
\end{lem}

\begin{proof}
Suppose $G$ is a minimal counterexample that contains such a vertex $v$.
Form \Emph{$G'$} from $G$ by deleting all horizontal edges of long bunches for which
$v$ is a parent.  It suffices to find an acyclic edge-coloring of $G'$ since, by
definition, we can extend it to $G$.  Let \Emph{$B$} be the longest bunch that has $v$
as a parent, and let $w$ be the other parent of this bunch.  Let $x$\aside{$v$,
$w$, $x$} be a bunch
vertex in $B$.  By minimality, we have an acyclic edge-coloring of $G'-x$; we can
greedily extend this to $G'-x+wx$, and we call this coloring $\vph$.  
We construct a set of colors $\C_{good}(v)$ as follows.  
Initially, let $\C_{good}(v)=[k]$.
For each color $\alpha$ used by $\vph$ on an edge $vp$ where $p$ is not a bunch vertex of
some bunch with $v$ as a parent, we do the following.
(Since $\nf+2\Mns\le 35$, we do this at most 35 times.)
Remove from $\C_{good}(v)$ both color $\alpha$ and either (i) all other colors used incident to $p$ or (ii) every
color used on an edge $vu$, whenever $u$ is a 2-vertex in $G'$ incident to an edge colored
with $\alpha$; for each color $\alpha$, we pick either (i) or (ii), giving
preference to the option that removes fewer colors from $\C_{good}$.  Since $\nf+2\Mns\le 35$,
the vertex $v$ is a parent for at most 35 bunches.  This is true because $x_0$
and $x_{t+1}$ are excluded from the bunch.  Thus, each application of (ii)
removes from $\C_{good}(v)$ at most 35 colors. Finally, we remove from
$\C_{good}(v)$ all
colors used on edges incident to $v$ that are in short bunches.  Since each
short bunch has at most 10 threads (by Lemma~\ref{lem4}), and $v$ is a parent for at most 35 bunches,
this removes from $\C_{good}(v)$ at most 350 colors. This completes
the construction of $\C_{good}(v)$.  Note that $\card{\C_{good}(v)}\ge
k-35(35)-35(10)=k-1575$.
Starting from $\vph$, we uncolor all 
edges incident to $v$ that used a color in $\C_{good}(v)$; these are all
edges of threads in bunches with $v$ as a parent.  We will use colors in
$\C_{good}(v)$ to recolor all of the uncolored edges, as well as $vx$ (first
with a proper coloring, and eventually with an acyclic coloring).  
This is the motivation behind our construction of $\C_{good}(v)$.

Suppose that $\vph(wx)$ is already used on some edge $vy$ in bunch $B$.
To avoid creating any 2-colored cycles through $x$,
it suffices to color $vx$ with any color in
$\C_{good}(v)\setminus\{\vph(wx),\vph(wy)\}$, which is easy.  So assume
$\vph(wx)$ is not used on any edge $vy$ in $B$.
(The hardest case is when $\vph(wx)$ is used on some edge
incident to $v$ leading to a non-bunch vertex.  This case motivates most of our
effort, so the reader will do well to keep it in mind.)
Our goal is to find some color, say $\alpha$, other than $\vph(wx)$, such that
$\alpha\in \C_{good}(v)$ and $\alpha$ is already used on an edge $wy$ of $B$. 
Given such an $\alpha$, we use it to color $vx$, and color $vy$ with some color
in $\C_{good}(v)\setminus\{\vph(wx),\alpha\}$.  This ensures that
each of $vx$ and $wx$ will never appear in a 2-colored cycle, no matter how we
further extend the coloring.  Such an $\alpha$ exists by the Pigeonhole
principle, because
$\mbox{length}(B)+\card{\C_{good}(v)}\ge k+2$.  We defer the computation proving
this to the end of the proof.  Now we extend our coloring to a proper (not
necessarily acyclic) $k$-edge-coloring of $G'$, using colors of $\C_{good}(v)$
on the uncolored edges.  This is easy by Hall's Theorem, since each edge has
only one color forbidden: the one already used incident to its endpoint of
degree 2. 

Now we modify this proper edge-coloring to make it acyclic.  It is important to
note that any 2-colored cycle must pass through $v$.  Further, it must use some
edges $e_1,e_2,e_3,e_4$, where $v$ is the common endpoint of $e_2$ and $e_3$ and
the common endpoints of edges $e_1$ and $e_2$ and of edges $e_3$ and $e_4$ are
both 2-vertices (this follows from our construction of $\C_{good}$). 
Suppose that such a 2-colored cycle exits, say with colors $\beta_1,\beta_2$.
One of these colors must be in $\C_{good}(v)$, since the 2-colored cycle did not
exist before assigning these colors; say it is $\beta_1$.  Suppose that a
second such 2-colored cycle exists, with colors $\gamma_1,\gamma_2$; by
symmetry, assume that $\gamma_1\in \C_{good}$.  To fix both cycles, we swap
colors $\beta_1$ and $\gamma_1$ on the edges incident to $v$ where they are
used.
We repeat this process until we have only at most one 2-colored cycle through
$v$.  Suppose we have one, with edges colored $\beta_1,\beta_2$ (and
$\beta_1\in \C_{good}(v)$); when we state the colors on edges of a thread, we
always start with the edge incident to $v$.  Now we look for some other thread 
with edges colored $\gamma_1,\gamma_2$ (and $\gamma_1\in \C_{good}(v)$) such
that no thread incident to $v$ has edges colored $\gamma_2,\beta_1$.  If we find
such a thread, then we swap colors $\beta_1$ and $\gamma_1$ on the edges
incident to $v$ where they appear, and this fixes the 2-colored cycle.  Since
$v$ is a parent in at most 35 bunches,
at most 35 incident threads have edges colored $\gamma_2,\beta_1$, for some
choice of $\gamma_2$.  Further, for each choice of $\gamma_2$, $v$ has at most
35 incident threads colored $\gamma_1,\gamma_2$, for some choice of $\gamma_1$.
Thus, at most $35^2=1225$ of these threads are forbidden.  
Recall from above that $\card{\C_{good}(v)}\ge k-1575$.  Now we have the desired
thread incident to $v$ since $d(v)-1225-1575>0$.  Thus, we can recolor the edge
colored $\beta_1$ to get an acyclic edge-coloring of $G'$, as desired.

Now we prove that $\mbox{length}(B)+\card{\C_{good}(v)}\ge k+2$. 
Note that $k-\card{C_{good}(v)}+2\le 5\nf+\Mns(\nf+\Mns+1-s)+10s+2$, where
\Emph{$s$} is the number of short bunches with $v$ as a parent.  This is
because each short bunch causes us to remove at most 10 colors, each vertex
counted by $\nf$ causes us to remove at most 5 colors, and each counted by
$\Mns$ causes us to remove at most $\nf+\Mns+1-s$ colors.  We must show that the
right side of the latter inequality is at most $\mbox{length}(B)$.  In fact,
we will show that it is no more than the average length of the long bunches
(rounded up).  Since the number of bunches is at most $\nf+\Mns$, we want 
the following inequality to hold.  On the left, the numerator is a lower bound
on the number of vertices in long bunches, and the denominator is an upper bound
on the number of long bunches.  The right side comes from the previous
inequality.
\begin{align*}
\frac{d(v)-(\nf+\Mns+10s)}{\nf+\Mns-s} >
5\nf+\Mns(\nf+\Mns+1-s)+10s)+1,
\end{align*}
which is implied by 
\begin{align*}
d(v)\ge (5\nf+\Mns(\nf+\Mns+1-s)+10s+1)(\nf+\Mns-s+1).
\end{align*}
Since $\nf+2\Mns\le 35$, it suffices to have
\begin{align*}
d(v)\ge (5(35-2\Mns)+\Mns((35-2\Mns)+\Mns+1-s)+10s+1)(35-\Mns-s+1 ).
\end{align*}
If we maximize the right side over all integers $\Mns$ and $s$ such that $0\le
\Mns\le 17$ and $0\le s\le 35-\Mns$ (using nested For loops, for example), then
we get 8680.
\end{proof}

\begin{lem}
\label{lem6}
Configuration (C4) cannot appear in a minimal counterexample $G$.
That is, $G$ cannot contain a very big vertex $v$ such that $\nf+2\Mns\le
141415$, where $\nf$ and $\Mns$ denote the numbers of $5^-$-neighbors and
$6^+$-neighbors of $v$ that are in no bunch with $v$ as a parent.
\end{lem}

\begin{proof}
Most of the proof is identical to that of Lemma~\ref{lem4}, that (C3) cannot
appear in a minimal counterexample.  The only difference is our argument showing
that $\mbox{length}(B)+\card{\C_{good}(v)}\ge k+2$, which we give now.
As in the previous lemma, it suffices to have
\begin{align*}
d(v)\ge (5\nf+\Mns(\nf+\Mns+1-s)+10s+1)(\nf+\Mns-s+1 ).
\end{align*}
By hypothesis, we have 
$\nf\le 141415-2\Mns$.
Now substituting for $\nf$, we get that it suffices to have
\begin{align}
\label{C4-bound}
d(v)\ge
&~(5(141415-2\Mns)+\Mns((141415-2\Mns)+\Mns+1-s)+10s+1)((141415-2\Mns)+\Mns-s+1
)\nonumber \\
= &~
\Mns^3+2\Mns^2s-282822\Mns^2+{\Mns}s^2-282832{\Mns}s \nonumber \\
&+19996363820\Mns-10s^2+707084s+99991859616.
\end{align}
We must upper bound the value of \eqref{C4-bound} over the region where $0\le
\Mns\le 70707$ and $0\le s\le \nf+\Mns-1\le 141415-\Mns$.  Since this domain is
much larger than in the previous lemma, we relax the integrality constraints and
solve a multivariable calculus problem.  The only critical point for this
function is outside the domain, so it suffices to find the maximum along the
boundary.  
This occurs when $s=0$ and $\Mns\approx47134$; the value is
approximately $4.19*10^{14}$.
Recall that $v$ is very big, so we have $d(v)\ge \Delta-4(8680)$.  Since we
need $d(v)\ge 4.19*10^{14}$, it suffices to require that $\Delta\ge
4.2*10^{14}$.  This completes the proof.
\end{proof}

\section*{Acknowledgments}
Thanks to Xiaohan Cheng for proposing acyclic
edge-coloring as a possible problem for the 2016 Rocky Mountain-Great Plains
Graduate Research Workshop in Combinatorics, and thanks to the organizers for
posting the problem write-ups on the website.  Reading that problem
proposal got me started on this research.
Thanks to Howard Community College for their hospitality during the research and
writing of this paper.  The idea of {bunches}, which was crucial to these
results, comes from two papers of Borodin, Broersma, Glebov, and van den
Heuvel~\cite{BBGvdH1,BBGvdH2}.  I am especially grateful that they
prepared English versions of these papers (which originally appeared in
Russian). The main idea in the proof of the structural lemma is that each
$5^-$-vertex is ``sponsored'' by its big neighbors.  This was inspired by a
paper with Marthe Bonamy and Luke Postle~\cite{BCP}. 
Thanks to Beth Cranston for helpful discussion.  Thanks to an anonymous
referee, whose careful reading caught errors and led to numerous small
improvements in the presentation.
Finally, thanks to the National Security Agency for partially
funding this research, under grant H98230-15-1-0013.

\bibliographystyle{plain}
{\scriptsize
\bibliography{GraphColoring}}

\end{document}